\documentclass[twoside,a4paper,reqno]{amsart}
\usepackage{style}
\usepackage[colorlinks=true, breaklinks=true, urlcolor=webbrown, linkcolor=RoyalBlue, citecolor=webgreen, backref=page]{hyperref}
\usepackage{verbatim} %

\begin{document}

\title[Orthogonal polynomials]{From local to global asymptotic behaviour\\ of orthogonal polynomials}
\author{R.\,Bessonov} 
\author{A.\,Nicolau}

\address{
R.\,Bessonov: Faculty of Mathematics and Physics\\ 
University of Ljubljana\\  
and Institute of Mathematics, Physics and Mechanics\\
Jadranska ulica 19, 1000 Ljubljana}
\email{roman.bessonov@fmf.uni-lj.si}

\address{A.\,Nicolau: Departament de Matem{\`a}tiques\\
Universitat Aut\`onoma de Barcelona\\ and Centre de Recerca Matem{\`a}tica \\08193 Barcelona}
\email{artur.nicolau@uab.cat}


\begin{abstract}
Let $\{\phi^*_n\}$ be the sequence of reflected orthogonal polynomials on the unit circle $\T$ generated by a measure $\mu$ of Szeg\H{o} class, and let $D_{\mu}$ be the Szeg\H{o} function of $\mu$. We prove the uniform Ces\`aro asymptotics 
$$
\sup_{z \in \Gamma_\zeta}\Biggl(\frac{1}{n}\sum_{k = 0}^{n-1}\Bigl||\phi_k^*(z) D_{\mu}(z)|^2 - 1\Bigr|\Biggr) \to 0, \qquad n \to \infty,
$$ 
for almost all Stolz angles $\Gamma_{\zeta}$, $\zeta\in \T$. This extends a well-known asymptotic result of M\'at\'e, Nevai, and Totik (1991) from the local scale $O(1/n)$ near $\T$ to the global scale $O(1)$. We also study asymptotic behavior of arguments of orthogonal polynomials and extend a classical theorem due to Grenander and Szeg\H{o} using a new technique. As an application, we derive global asymptotic results for polynomial reproducing kernels under various assumptions on the orthogonality measure.            
\end{abstract}

\maketitle

\section{Introduction}
Let $\mu$ be a probability measure on the unit circle $\T = \{z: \; |z| = 1\}$. Consider its Radon-Nikodym decomposition, $\mu = w\,dm +\mus$. Here and in what follows $m$ is the Lebesgue measure on $\T$ normalized by $m(\T) = 1$, the function $w$ belongs to $L^1(\T) = L^1(m)$, and $\mus$ is the singular part of $\mu$. The measure $\mu$ is said to belong to the Szeg\H{o} class $\szc$ if $\log w \in L^1(\T)$, or, equivalently, 
\begin{equation}\label{eqSz}
\int_{\T}\log w\, dm > - \infty.
\end{equation}
Measures of Szeg\H{o} class play a prominent role in the theory of orthogonal polynomials \cite{Szbook}, \cite{Simonbook1}, \cite{SimonDes}. With each $\mu \in \szc$ one can associate the unique outer function $D_\mu$ in the open unit disk $\D = \{z: |z| < 1\}$ satisfying $|D_\mu|^2 = w$ almost everywhere on $\T$ and such that $D_\mu(0) > 0$. This function is called the Szeg\H{o} function of $\mu$. Let $\{\phi_n\}$ be the sequence of orthogonal polynomials in $L^2(\mu)$, $\deg \phi_n = n$, and let $\phi_{n}^{*} $ be the corresponding reflected polynomials defined by $\phi_{n}^{*} (z) = z^n \ov{\phi_n(1/\bar z)}$. We use the standard normalization
\begin{equation}\label{eqonb}
\|\phi_{n}\|_{L^2(\mu)} = 1, \qquad \phi_{n}^{*}(0) > 0, \qquad n \ge 0.
\end{equation}
Szeg\H{o} theorem determines the asymptotic behavior of reflected polynomials $\phi_n^*$ in the open unit disk $\D$. It says that
$$
\sup_{n \ge 0} \phi_{n}^{*}(0) < \infty \quad \Longleftrightarrow \quad \mu \in \szc,
$$
and, moreover, for every $\mu \in \szc$ we have
\begin{equation}\label{szthm}
\lim_{n \to \infty} \phi^*_n(z) D_{\mu}(z) = 1, 
\end{equation}
uniformly on compact subsets of $\D$. See Theorem 2.4.1 and Theorem 2.7.15 in Simon book \cite{Simonbook1} for the proof. One of the most studied questions in the theory of orthogonal polynomials asks under which assumptions on the orthogonality measure $\mu$, relation \eqref{szthm} holds almost everywhere on the unit circle $\T$. Let us briefly describe known results (the reader familiar with the history of the problem can go directly to Section \ref{s2} for the results of the present paper). 

\medskip

\subsection{Regular measures, uniform convergence}\label{s11}
When studying uniform convergence of $\{\phi^*_n\}$ on $\T$, it is natural to require smoothness properties on the measure $\mu$. Let us state apparently the best known result in this direction. Let $\mu \in \szc$ be such that $\mu = w\,dm$ and $\log w$ is a continuous function in $\T$. Assume its modulus of continuity $\omega (t)= \sup \{|\log w (\xi) - \log w(z)| : |\xi - z| \le t \}$ satisfies the following Dini condition
\begin{equation} \label{dini}
\int_0^1 \frac{\omega (t)}{t}\,dt < \infty . 
\end{equation}
Then \eqref{szthm} holds uniformly on $\T$. Geronimus and Golinski \cite{GG65} attributed this result to Grenander and Szeg\H{o} \cite{GS58}. A proof can be found in Golinski \cite{G67} and a local version, that is, considering the modulus of continuity in an arc, in Badkov \cite{Badkov79}. It is worth mentioning that the condition on the continuity of $\log w$ is not sufficient and \eqref{dini} is really needed. Actually Golinski \cite{G74} proved that there exist weights $w$ such that $\log w$ is a continuous function on $\T$ whose modulus of continuity satisfies $\omega (t) \log (t) \to 0$ as $t \to 0$, such that $\{\phi^*_n\}$ does not converge uniformly on $\T$.     

\subsection{Steklov problem for weights uniformly separated from zero}

Given $0<\delta < 1$, a probability measure $\mu = w\,dm$ is in the Steklov class $S_\delta$ if $w \geq \delta$ at almost every point in $\T$. To emphasize the role of the measure $\mu$, the orthonormal polynomials with respect to the measure $\mu$ will  be denoted by $\varphi_n (z, \mu)$. In 1921, Steklov raised the conjecture of deciding if $\sup_n |\varphi_n (z, \mu)| < \infty$ for any $z \in \T$ provided $\mu \in S_\delta$. This conjecture was solved in the negative by Rahmanov \cite{R81} leading to the problem of finding the right asymptotics of $\|\varphi_n (z, \mu)\|_\infty = \sup \{|\varphi_n (z, \mu) |: z \in \T\}$. Consider $M(n , \delta) = \sup \{\|\varphi_n (z, \mu)\|_\infty : \mu \in S_\delta\}$. Aptekarev, Denisov and Tulyakov \cite{ADT} proved the remarkable result that fixed $0<\delta<1, $ $M(n , \delta)$ is comparable to $\sqrt{n}$. Later Denisov \cite{D18} considered positive measures $\mu = w\,dm$ with $w,\,1/w \in L^\infty (\T) $ and proved that there exists a constant $C>0$ such that $\sup \{ \|\varphi_n (z, \mu)\|_\infty : 1 \le w \le 1 + \varepsilon \} \le C n^{C \varepsilon}$ if $0 < \varepsilon < 1$ and $\sup \{ \|\varphi_n (z, \mu)\|_\infty : 1 \le w \le T \} \le C n^{ 1/2 - C/T}$ if $T>2$. Moreover these estimates are sharp up to the value of the constant $C$. See also  Ambroladze \cite{Amb91} for continuous weights, Denisov and Rush \cite{DR17} for $\rm{BMO}$ weights, and Alexis, Aptekarev, Denisov \cite{AAD} for Muckenhoupt $A_p$ weights.

\subsection{Averaged convergence, general Szeg\H{o} measures}
In 1991, M\'at\'e, Nevai, and Totik \cite{MNT} proved a fundamental result on Christoffel functions generated by general measures $\mu$ of Szeg\H{o} class $\szc$. In the equivalent language of orthogonal polynomials, it establishes Ces\`aro convergence
\begin{equation}\label{eqMNT}
\lim_{n \to \infty}\frac{1}{n}\sum_{k=0}^{n-1}|\phi_k^*(z) D_{\mu}(z)|^2 = 1
\end{equation}
for Lebesgue almost every point $z$ on the unit circle $\T$.  In 2008, Findley \cite{F08} showed that \eqref{eqMNT} holds almost everywhere on an open arc $I \subset \T$ if we assume that the measure $\mu$ is regular in the sense of Ullman and the {\it local} Szeg\H{o} condition holds, i.e., $\log w \in L^1(I)$. Originally, the study of asymptotic behaviour of Christoffel functions was motivated by the problem of Ces\`aro summability of Fourier series generated by orthogonal polynomials \cite{MN80}, in particular, by the previous work of Freud \cite{Fr52}.  Surprisingly, \eqref{eqMNT} turned out to be very important in the seemingly unrelated area of universality properties of orthogonal polynomials. The first application of \eqref{eqMNT} in this context is due to Lubinsky \cite{Lubinsky}, it fast became a standard method. Totik \cite{Totik} discusses a general idea of usage of  \eqref{eqMNT} in universality problems and gives another proof of Findley's result. A recent breakthrough by Eichinger, Luki\'c, and Simanek~\cite{ELS} showed that Szeg\H{o} condition \eqref{eqSz} is not really needed for universality, this phenomenon is much more general. Still, under the Szeg\H{o} condition we know {\it the scale} at which universality holds ($\sim 1/n$), and this knowledge follows directly from \eqref{eqMNT}, see the discussion after Theorem 1.2 in \cite{ELS}. In the Szeg\H{o} case we also know a qualitative estimate for {\it the rate of convergence} of universality limits, see \cite{Bes24}. This information is not available in the general non-Szeg\H{o} case due to a usage of compactness in \cite{ELS}. 

\medskip

\subsection{Assumptions on the side of recurrence coefficients}
In the next portion of convergence results, assumptions on the orthogonality measure $\mu$ are formulated in terms of the so-called recurrence coefficients. Let us recall that orthogonal polynomials generated by a probability measure $\mu$ with an infinite support on $\T$ satisfy the following recurrence relation:
\begin{equation}\label{eqreq}
\phi_{n+1}(z) = \frac{z \phi_n(z) - \ov{a_n}\phi_n^*(z)}{\sqrt{1-|a_n|^2}}, \quad n \ge 0, 
\end{equation}
for all $z \in \C$ and some numbers $a_n \in \D$, $n \ge 0$, that are called recurrence coefficients of $\mu$. It is also known that any sequence $\{a_n\}_{n \ge 0} \subset \D$ is the sequence of recurrence coefficients of a unique probability measure $\mu$ with an infinite support on $\T$. See Sections 1.5, 1.7 in \cite{Simonbook1} for the proofs of these facts. Heuristically, small recurrence coefficients $\{a_n\}$ give rise to ``regular'' measures $\mu$, Szeg\H{o} functions $D_\mu$, and reflected orthogonal polynomials $\phi_n^*$.  For example, we have $\mu = m$, $D_\mu = 1$, $\phi_n^* = 1$, $n \ge 0$, for identically zero recurrence coefficients $\{a_n\}_{n \ge 0}$. It is therefore natural to ask about the almost sure convergence property 
\begin{equation}\label{convc}
\lim_{n \to \infty} \phi^*_n(z) D_{\mu}(z) = 1 \;\; \mbox{for Lebesgue almost every }z \in \T,
\end{equation}
in terms of the size of the sequence $\{a_n\}_{n \ge 0}$. Using formula $(2.4.35)$ in Simon \cite{Simonbook1} for polynomials $P = \phi_k$, one  can see that  
$$
(\ov{D_{\mu}(0)^{-1}} D_{\mu}(z)^{-1}, \phi_k)_{L^2(w\,dm)} = \ov{\phi_k(0)}, \qquad k \ge 0. 
$$
Denote by $S$ a subset of $\T$ such that $m(S) = 1$, $\mus(S) = 0$, and let $\chi_{S}$ be the characteristic function of $S$. It is known that $\chi_{S}D_{\mu}^{-1}$ can be approximated in $L^2(\mu)$ by analytic polynomials, see formula $(2.4.34)$ in \cite{Simonbook1}. It follows that $\chi_{S}D_{\mu}^{-1}$ belongs to the closed linear span of the orthonormal sequence $\{\phi_k\}_{k\ge 0}$ in $L^2(\mu)$. Then basic Fourier analysis tells us that
\begin{equation}\label{eq144}
\chi_{S}(z)\ov{D_{\mu}(0)^{-1}}D_{\mu}(z)^{-1} = \sum_{n = 0}^{\infty} \ov{\phi_k(0)}\phi_k(z), \qquad z \in \T,
\end{equation}
in the sense of convergence in $L^2(\mu)$. Moreover, we have 
$$
\ov{\phi_n^*(0)}\phi_n^*(z) - \ov{\phi_n(0)}\phi_n(z)= \sum_{k = 0}^{n-1} \ov{\phi_k(0)}\phi_k(z), \qquad z \in \C, \quad n\ge 1,
$$
see $(2.2.42)$ in \cite{Simonbook1}. From here and the Szeg\H{o} asymptotic relation \eqref{szthm} at $z = 0$ we see that \eqref{convc} holds if and only if the orthonormal series 
\begin{equation}\label{ortho-series}
\sum_{k = 0}^{\infty} c_k \phi_k
\end{equation}
converges pointwise almost everywhere on the unit circle $\T$ for the special choice of Fourier coefficients $c_k = \ov{\phi_k(0)}$. The problem of pointwise convergence of orthogonal series was studied in much detail by Menshov. One of his most striking results~\cite{Menshov26} says that for each $\eps > 0$ and any square-summable sequence of coefficients $\{c_k\}_{k\ge 0}$ the condition  
$$
\sum_{k: \; 0 < |c_k| < 1} |c_k|^2 \left(\log\log\frac{1}{|c_k|}\right)^{2+\eps}\left(\log\frac{1}{|c_k|}\right)^{2} < \infty  
$$
implies the pointwise convergence of \eqref{ortho-series} almost everywhere on $\T$. In particular, the above condition holds if 
$$
\sum_{k\ge 0} |c_k|^p < \infty  
$$
for some $0 < p< 2$. Another theorem of Menshov \cite{Menshov23} (obtained independently by Rademacher \cite{Rad22}) says that condition
$$
\sum_{k \ge 1} |c_k|^2 \log^2 k < \infty, 
$$
is also sufficient for the pointwise convergence of \eqref{ortho-series} almost everywhere on $\T$. Using the fact that the product $\prod_{k = 0}^{\infty}(1 - |a_k|^2)$ converges to a positive number for every $\mu \in \szc$ (this is another version of the Szeg\H{o} theorem, see Theorem 2.7.14 in \cite{Simonbook1}), and the formula 
$$
\phi_{k+1}(0) = -\frac{\ov{a_k}}{\sqrt{\prod_{j=0}^{k}(1-|a_j|^2)}}, \qquad k \ge 0,
$$
from Section 1.5 in  \cite{Simonbook1}, we see from Menshov theorems that any of the assumptions 
\begin{gather}
\sum_{k: \; a_k \neq 0} |a_k|^2 \left(\log\log\frac{1}{|a_k|}\right)^{2+\eps}\left(\log\frac{1}{|a_k|}\right)^{2} < \infty \;\;\mbox{ for some }\;\; \eps > 0, \label{eq141} \\
\sum_{k \ge 0} |a_k|^p< \infty \;\;\mbox{ for some }\;\; 0 < p < 2,\\
\sum_{k \ge 1} |a_k|^2 \log^2 k < \infty, \label{eq143}
\end{gather}
implies \eqref{convc}. Similarly, yet another Menshov result \cite{Menshov26} implies that assumption
$$
\sum_{k \ge 3} |a_k|^2 (\log\log k)^2 < \infty
$$
guarantees the Ces\`aro convergence
$$
\lim_{n \to \infty}\frac{1}{n}\sum_{k=0}^{n-1}\phi_k^*(z) D_{\mu}(z) = 1
$$
almost everywhere on the unit circle $\T$. The reader can find more information on Menshov's works and their subsequent development in the review paper by Ul'yanov~\cite{Ul92}.

\medskip

\subsection{A theorem by Poltoratski and Denisov's counterexample} Theory of orthogonal polynomials on the unit circle (OPUC) is strongly related to the spectral theory of self-adjoint differential operators with simple spectrum. Let us mention two important papers in that field related to our work. In \cite{Polt24},  a nonlinear Carleson theorem for Dirac operators is proved, i.e., a continuous version of the statement  
$$
\lim_{n \to \infty} |\phi^*_n(z) D_{\mu}(z)| = 1 \;\; 
\mbox{for every $\mu \in \szc$ and Lebesgue almost every $z \in \T$.}
$$
An error in \cite{Polt24} was found by Gevorg Mnatsakanyan. A new variant of the proof is avaliable in the form of arXiv preprint \cite{Polt24preprint}.

\medskip

In recent preprints \cite{Den26b}, \cite{Den26}, Denisov provides a counterexample to a strong version of the nonlinear Carleson theorem. He actually constructs \cite{Den26b} a measure $\mu \in \szc$ such that \eqref{convc} fails.

\medskip

\section{Main results}\label{s2}
Consider a signed Borel measure $\mu$ of finite variation on the unit circle $\T$. Let us denote by $\Pp(\mu, z)$ the Poisson integral of $\mu$ at $z \in \D$, given by 
$$
\Pp(\mu, z) = \int_{\T}\frac{1-|z|^2}{|1 - \bar \xi z|^2}\,d\mu (\xi), \qquad z \in \D.
$$
For $u \in L^1(\T)$ and $\mu = u\,dm$, we set $\Pp(u, z) = \Pp(\mu, z)$.  Consider now a probability measure $\mu = w\,dm + \mus$ on $\T$ that belongs to the Szeg\H{o} class $\szc$. Recall that the latter means $\log w \in L^1(\T)$. Our analysis of asymptotic behaviour of orthogonal polynomials will be largely connected with the following entropy function of $\mu$:
$$
\K(\mu, z) = \log\Pp(\mu, z) - \Pp(\log w, z), \qquad z \in \D.
$$
By Jensen inequality, $\K(\mu, z) \ge 0$ for every $z \in \D$. Take $\zeta \in \T$ and consider the Stolz angle with vertex at $\zeta$, i.e., the convex hull 
$$
\Gamma_{\zeta} = \conv\bigl(\{z \in \C:\; |z| \le 1/2\},\, \{\zeta\}\bigr).
$$
Here the constant $1/2$ can be replaced by any constant $r \in (0, 1)$ without essential changes in the forthcoming statements and their proofs. 
Properties of the Poisson kernel imply that for every $\mu \in \szc$ we have
\begin{equation}\label{eq21}
\lim\limits_{\substack{z \to \zeta,\\ z \in \Gamma_\zeta}}\K(\mu, z) = 0
\end{equation}
for Lebesgue almost every $\zeta \in \T$. Sometimes we will also require more regularity from the measure $\mu$ than just the membership in the Szeg\H{o} class $\szc$. Take $\zeta \in \T$ and set $z_{n}(\zeta) = (1-2^{-n})\zeta$ for $n \ge 0$. Note that $\{z_{n}(\zeta)\}$ is a sequence in $\Gamma_\zeta$ approaching the point $\zeta$ exponentially fast. As we will check  (see the discussion after Proposition \ref{p31} below), both conditions
\begin{align}
&\sum_{n \ge 0} \K(\mu, z_n(\zeta)) < \infty, \label{eq22}\\
&\sum_{n \ge 0} \sqrt{\K(\mu, z_n(\zeta))} < \infty, \label{eq23} 
\end{align}
are weaker than the Dini condition \eqref{dini} or its local version on an arc containing~$\zeta$. In particular, assumptions \eqref{eq22}, \eqref{eq23} do not force $w$ to be continuous near $\zeta$. Observe also that \eqref{eq22} and \eqref{eq23} are the discrete versions of conditions
$$
\int_0^1 \frac{\K(\mu, r \zeta)}{1-r}\,dr < \infty , \qquad \int_0^1 \frac{\sqrt{\K (\mu, r \zeta)}}{1-r}\,dr < \infty.
$$
Our main results are the following theorems.
\begin{Thm}\label{thm1}
Let $\mu \in \szc$ and $\zeta \in \T$ be such that $|D_\mu|$ has a non-zero finite non-tangential limit at $\zeta$. If \eqref{eq21} holds at $\zeta$, then
\begin{equation}\label{thm-eq1}
\sup_{z \in \Gamma_\zeta}\Biggl(\frac{1}{n}\sum_{k = 0}^{n-1}\Bigl||\phi_k^*(z) D_{\mu}(z)|^2 - 1\Bigr|\Biggr) \to 0, \qquad n \to \infty.
\end{equation}
In particular, \eqref{thm-eq1} holds for Lebesgue almost every  $\zeta \in \T$ for any $\mu\in\szc$.
\end{Thm}
\begin{Thm}\label{thm2}
Let $\mu \in \szc$ and $\zeta \in \T$ be such that $D_\mu$ has a non-zero finite non-tangential limit at $\zeta$. If \eqref{eq22} holds at $\zeta$, then
\begin{equation}\label{thm-eq2}
\sup_{z \in \Gamma_\zeta}\Biggl(\frac{1}{n}\sum_{k = 0}^{n-1}\bigl|\phi_k^*(z) D_{\mu}(z) - 1\bigr|\Biggr) \to 0, \qquad n \to \infty.
\end{equation}
\end{Thm}
\begin{Thm}\label{thm3}
Let $\mu \in \szc$ and $\zeta \in \T$ be such that \eqref{eq23} holds at $\zeta$. Then
\begin{equation}
    \label{noou}
    \sup_{z \in \Gamma_\zeta}|\phi^*_n(z) D_{\mu}(z) -1| = 0, \qquad n \to \infty.
\end{equation}
\end{Thm}

\begin{figure}
\centering
\begin{subfigure}{.5\textwidth}
  \centering
\includegraphics[width=.8\linewidth]{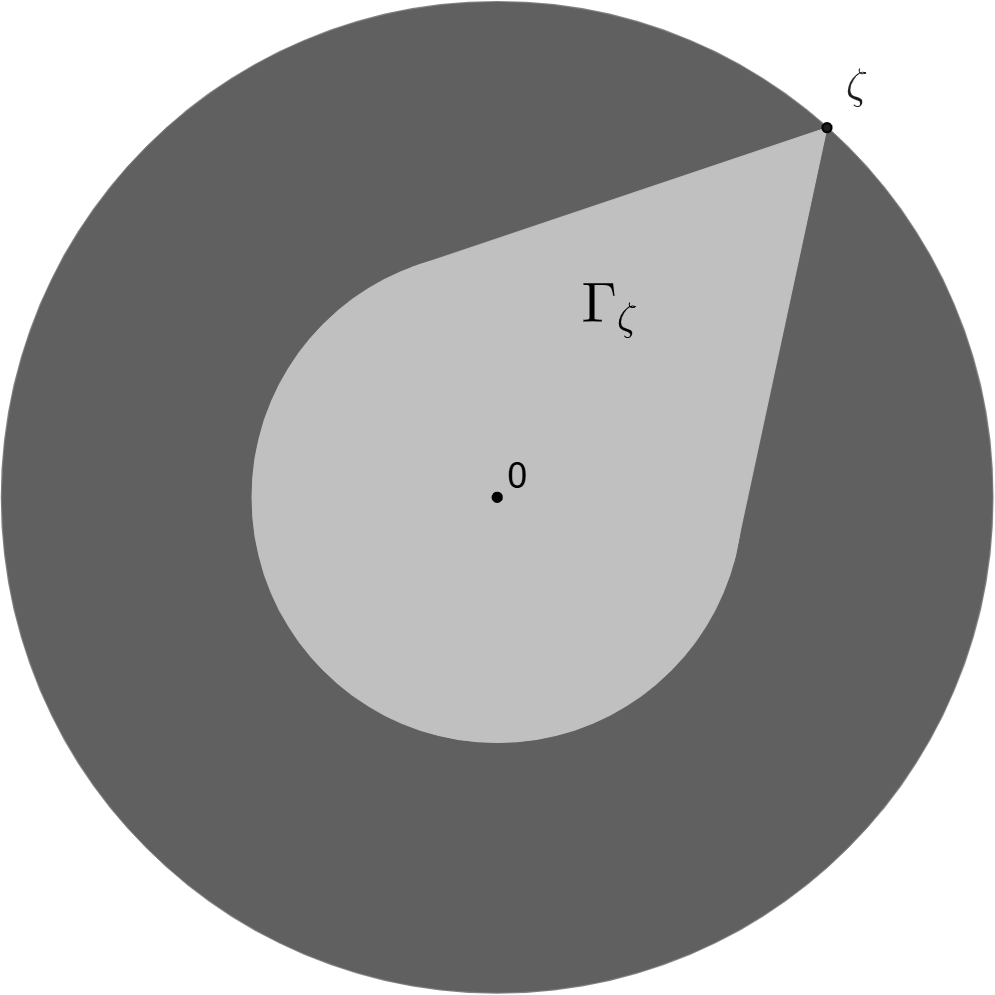}
  \caption{Stolz angle $\Gamma_\zeta$}
  \label{fig:sub1}
\end{subfigure}%
\begin{subfigure}{.5\textwidth}
  \centering
  \includegraphics[width=.8\linewidth]{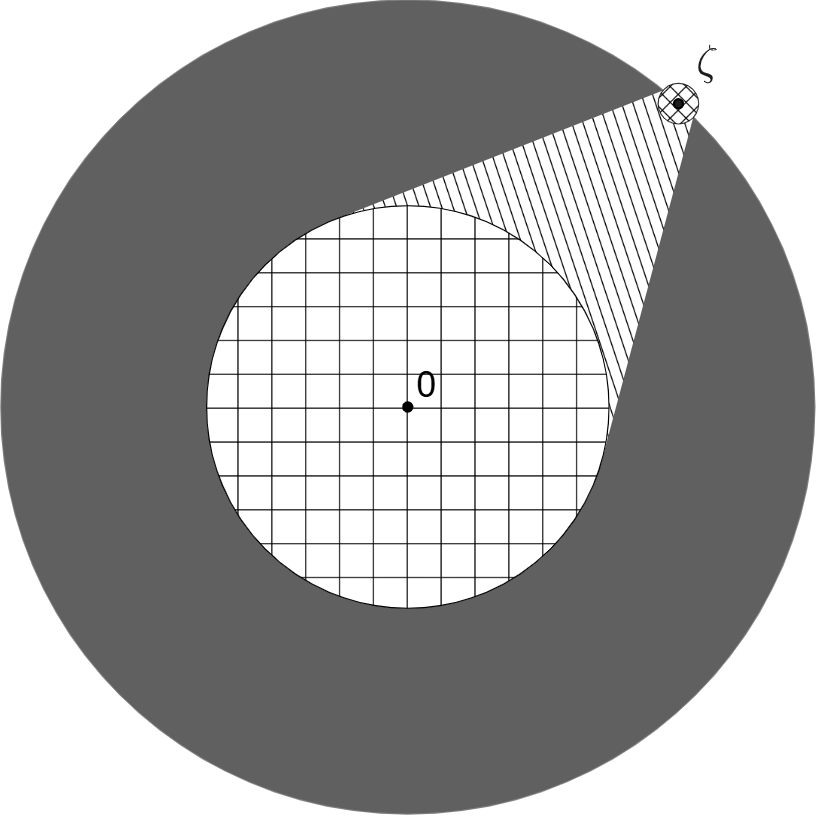}
  \caption{Tiling the domain of convergence}
  \label{fig:sub2}
\end{subfigure}
\caption{Different asymptotic methods work in different regions. On compact subsets of $\D$, the classical Szeg\H{o} argument applies. Near the  unit circle, at disks of radii $\sim 1/n$, one can use M\'at\'e, Nevai, and Totik approach. We prove the uniform Ces\`aro asymptotics in the whole Stolz angle $\Gamma_\zeta$. See also Theorem \ref{thm1-bis} for the convergence in the bigger domain shown in Figure \ref{fig:sub2}.}
\label{fig:main}
\end{figure}

\noindent Let us make some remarks on Theorems \ref{thm1}--\ref{thm3}. First, we would like to note that Theorem \ref{thm1} implies M\'at\'e, Nevai, and Totik theorem \cite{MNT}, see \eqref{eqMNT}. Indeed, the property
\begin{equation}\label{eq31}
\sup_{z \in \Gamma_\zeta}\Biggl(\frac{1}{n}\sum_{k = 0}^{n-1}\Bigl||\phi_k^*(z) D_{\mu}(z)|^2 - 1\Bigr|\Biggr)  \to 0, \qquad n \to \infty,
\end{equation}
is stronger than property 
\begin{equation}\label{eq32}
\frac{1}{n}\sum_{k = 0}^{n-1}|\phi_k^*(\zeta) D_{\mu}(\zeta)|^2 \to 1, \qquad n \to \infty,
\end{equation}
established in \cite{MNT} for almost all $\zeta \in \T$. Besides the fact that \eqref{eq31} holds globally in the Stolz angle $\Gamma_\zeta$, see Figure \ref{fig:main}, there is another important difference between \eqref{eq31}, \eqref{eq32}. Namely, \eqref{eq32} tells us that for $n$ large enough {\it the arithmetic mean} of the sequence $\{|\phi_k^*(\zeta) D_{\mu}(\zeta)|^2\}_{1 \le k \le n}$ is close to $1$, while \eqref{eq31} implies that {\it most of the members} of this sequence are close to $1$. 

\medskip

Our second remark concerns sharpness of Theorem \ref{thm1}. Based on construction in \cite{Den26b}, we conjecture that \eqref{thm-eq1} cannot hold with $|\phi_k^*(z) D_{\mu}(z)|^2$ replaced by $\phi_k^*(z)^2 D_{\mu}(z)^2$ for general measures of Szeg\H{o} class.

\medskip 

Let us now discuss assumptions \eqref{eq22}, \eqref{eq23} in Theorems \ref{thm2}, \ref{thm3}. Without trying to formulate the most general results, we would like to demonstrate which properties of $\mu$ can yield \eqref{eq22}, \eqref{eq23}. 

\begin{Prop}\label{p31}
Assume that $\mu = w\,dm$ is such that $u = \log w$ is bounded on $\T$ and its Fourier coefficients $\hat u(k) = (u, z^k)_{L^2(\T)}$ satisfy 
\begin{equation}\label{eq33}
\sum_{k\ge 1}|\hat u( k)|^2  \log k < \infty.
\end{equation}
Then \eqref{eq22} holds at Lebesgue almost every $\zeta \in \T$.
\end{Prop}

\begin{Prop}\label{p21}
Let $\mu = w\,dm$ be a probability measure in $\szc$ and let $\zeta \in \T$. Assume that $\log w$ is bounded at a neighborhood of $\zeta$ and 
\begin{equation}\label{eq34}
\int_{\T}\frac{|\log w(\xi) - \log w(\zeta)|^2}{|\xi - \zeta|}\,dm(\xi) < \infty.
\end{equation}
Then \eqref{eq22} holds at $\zeta$. 
\end{Prop}

\begin{Prop}\label{p11}
Let $\mu = w\,dm$ be a probability measure in $\szc$, let $\zeta \in \T$ and $I_n =  \{\xi \in \T :|\xi - \zeta| < 2^{-n} \}$, $n\ge 0$. Assume that $\log w$ is bounded in a neighborhood of $\zeta$ and 
\begin{equation}\label{eq341}
\sum_{n=1}^\infty \left( \frac{1}{|I_n|}\int_{I_n}\bigl|\log w - \log w(\zeta)\bigr|^2 \,dm \right)^{1/2}  < \infty . 
\end{equation}
Then \eqref{eq23} holds at $\zeta$.
\end{Prop}

Theorem \ref{thm3} and Proposition \ref{p11} imply the classical result of Grenander and Szeg\H{o} and its improvement by Badkov (see Section \ref{s11}). Indeed, for $t = |\xi - \zeta|$ we have
$$
|\log w(\xi) - \log w(\zeta)| \lesssim |w(\xi) - w(\zeta)| \leq \sup_{|z_1 - z_2| \le t} |w(z_1) - w(z_2)| = \omega(t)  
$$
if $w$ is positive and continuous near $\zeta$ and $t$ is small enough (as usual, we write $g_1 \lesssim g_2$ if $g_1 \le c g_2$ for a constant $c$ that does not depend on the parameters under consideration, i.e., in the above case, on $t$). Therefore, \eqref{eq341} is weaker than the classical Dini condition \eqref{dini}:
\begin{gather*}
\sum_{n=1}^\infty \left( \frac{1}{|I_n|}\int_{I_n}\bigl|\log w - \log w(\zeta)\bigr|^2 \,dm \right)^{1/2} \lesssim \sum_{n=1}^\infty \omega(|I_n|) \lesssim \\ \lesssim \sum_{n=1}^\infty \int_{2^{-n}}^{2^{-n+1}}\frac{\omega(t)}{t}\,dt \le \int_{0}^{1}\frac{\omega(t)}{t}\,dt < \infty,
\end{gather*}  
or its local version on an arc. Since our estimates for the remainders are qualitative, Theorem \ref{thm3} and Proposition \ref{p11} imply the uniform convergence $\phi^*_n D_{\mu} \to 1$ on any open arc of $\T$ where $w$ is positive and satisfies the Dini condition. 

\medskip

We would like to emphasize that assumptions in Propositions \ref{p31}-\ref{p11} (or, more generally, assumptions \eqref{eq21}--\eqref{eq23}) are given directly in terms of the measure $\mu$. Therefore, they are much easier to check than assumptions \eqref{eq141}--\eqref{eq143} in terms of recurrence coefficients when one starts from the orthogonality measure $\mu$. On the other hand, if one starts with recurrence coefficients of $\mu$ (such a situation occurs in direct spectral theory of differential operators with simple spectrum), then assumptions \eqref{eq141}--\eqref{eq143} are much easier than \eqref{eq21}--\eqref{eq23}. From this perspective,  \eqref{eq21}--\eqref{eq23} and \eqref{eq141}--\eqref{eq143}  complement each other. 

\medskip

Next, we apply our results to the study of the asymptotic behaviour of the polynomial reproducing kernels 
\begin{equation}\label{reprokernel}
k_{\mu, n}(z_1, z_2) = \frac{\ov{\phi_n^*(z_2)}\phi_n^*(z_1) - \ov{\phi_n(z_2)}\phi_n(z_1)}{1 - \ov{z_2} z_1} 
\end{equation}
generated by a measure $\mu \in \szc$. It is well-known that these kernels have the following universal asymptotic behaviour
\begin{gather}
k_{\mu, n}(z_1, z_2) = |D_{\mu}^{-1}(\zeta)|^2 \frac{1 - \bar z_2^{n} z_1^{n}}{1 - \bar z_2 z_1}\bigl(1 + r_n(z_1, z_2)\bigr), \label{universality0} \\ 
\lim_{n \to \infty}\sup \{ \bigl|r_n(z_1, z_2)\bigl| \colon z_{1,2} \in \mathbb{C},\, |z_{1,2} - \zeta| \le A/n\} = 0, \label{universality}
\end{gather}
for every $A \ge 0$ and almost every point $\zeta \in \T$. Moreover, a bound for the rate of convergence in \eqref{universality} is known, see \cite{Bes24}. Let us write $g_1 \asymp g_2$ for two functions $g_1$, $g_2$ if $g_1 \lesssim g_2$ and $g_2 \lesssim g_1$. Recall that the hyperbolic distance in $\D$ is defined by
\begin{equation}\label{eqHD}
d_{H}(z_1, z_2) = \frac{1}{2}\log\frac{1+\rho(z_1, z_2)^2}{1-\rho(z_1, z_2)^2},  \qquad \rho(z_1, z_2) = \left|\frac{z_1 - z_2}{1 - \ov{z_2}z_1}\right|.
\end{equation}
Note that $d_{H}(z_1, z_2) \lesssim 1$ for two points $z_1$, $z_2$ in some Stolz angle $\Gamma_{\zeta}$ if and only if $1 - |z_1| \asymp 1 - |z_2|$. We prove the following extension of the local universality relation \eqref{universality0}.
\begin{Thm}\label{thm4}
Let $\mu \in \szc$ and $\zeta \in \T$ be such that $D_\mu$ has a non-zero finite non-tangential limit at $\zeta$. Take $A \ge 0$. Define the function $r_n$ in $\D \times \D$ by
\begin{equation}\label{universality0-thm4}
k_{\mu, n}(z_1, z_2) = \ov{D_{\mu}^{-1}(z_2)}D_{\mu}^{-1}(z_1) \frac{1 - \bar z_2^{n} z_1^{n}}{1 - \bar z_2 z_1}\bigl(1 + r_n(z_1, z_2)\bigr).
\end{equation}
We have
\begin{align}
\lim_{n \to \infty} \sup_{\substack{z_1, z_2 \in \Gamma_{\zeta},\\ d_H(z_1, z_2) \le A}}\bigl|r_n(z_1, z_2)\bigl| = 0, \quad 
&\mbox{if \eqref{eq21} holds,} \label{universality-thm4}\\   
\lim_{n \to \infty}\sup_{z_1, z_2 \in \Gamma_{\zeta}}\frac{1}{n}\sum_{k=0}^{n-1}\bigl|r_k(z_1, z_2)\bigl|= 0, \quad 
&\mbox{if \eqref{eq22} holds,} \label{universality-thm4-2}\\
\lim_{n \to \infty} \sup_{z_1, z_2 \in \Gamma_{\zeta}}\bigl|r_n(z_1, z_2)\bigl| = 0, \quad 
&\mbox{if \eqref{eq23} holds.} \label{universality-thm4-3}
\end{align}
In particular, we have \eqref{universality-thm4} for almost every $\zeta \in \T$ for any measure $\mu\in\szc$.
\end{Thm}

\section{Preliminaries}
In this section we recall the basic theory of orthogonal polynomials and collect some results from \cite{Bes24}. Given a probability measure $\mu$ on $\T$, its Schur function $f$ is defined by the relation 
\begin{equation}\label{eq1}
\frac{1+zf(z)}{1-zf(z)} = \int_{\T}\frac{1 + \bar \xi z}{1 - \bar \xi z}\,d\mu (\xi), \qquad z \in \D.
\end{equation}
Taking the real part in \eqref{eq1}, we get
\begin{equation}\label{eq55}
\frac{1 - |zf(z)|^2}{|1 - zf(z)|^2} = \int_{\T}\frac{1 - |z|^2}{|1 - \bar\xi z|^2}\,d\mu (\xi) = \Pp(\mu, z), \qquad z \in \D.
\end{equation}
Since expressions in \eqref{eq55} are positive in $\D$, we have $|zf(z)| < 1$ for every $z \in \D$. In other words, $zf$ is an analytic mapping from $\D$ into itself. By Schwarz lemma and the maximal modulus principle, the same is true about $f$ (with the exceptional case when $f$ is a unimodular constant, it corresponds to the case where $\mu$ is the point mass measure concentrated at a point of $\T$). Conversely, any analytic mapping $f$ from $\D$ into itself leads to a probability measure $\mu$ on $\T$ defined by formula \eqref{eq1}. It is not difficult to check that in this correspondence the finite Blaschke products generate probability measures $\mu$ supported on finite subsets of~$\T$. In what follows we will always exclude this case from consideration assuming that $\mu$ is not a finite linear combination of point masses.

\medskip

Consider a pair $\mu = w \, dm + \mus$, $f$ related by  \eqref{eq1}.  Fatou's theorem (see Theorem~5.3  in Section~I.5 in Garnett \cite{Garnett}) implies that for Lebesgue almost every $\xi \in \T$ we have
\begin{equation}\label{eq56}
\frac{1 - |f(\xi)|^2}{|1 - \xi f(\xi)|^2} = w(\xi) .
\end{equation}
Note that $\Re(1 - z f(z)) \ge 0$ for $z \in \D$, hence the function $1 - z f$ is outer in $\D$ (see Corollary 4.8 in Section II.4 in Garnett \cite{Garnett}). Hence, we have $\log|1 - \xi f| \in L^1(\T)$ and
\begin{equation}\label{eq62}
\int_{\T}\log{|1 - \xi f(\xi)|^2}\,dm (\xi) = \log|1 - z f(z)|^2\Bigr\rvert_{z = 0} = 0.
\end{equation}
This formula and \eqref{eq56} imply
$$
\int_{\T}\log{(1 - |f|^2)}\,dm = \int_{\T}\log w\,dm, 
$$
where both sides are finite or equal to $-\infty$ simultaneously. In other words, we have $\mu \in \szc$ if and only if $\log(1-|f|^2) \in L^1(\T)$. 

\medskip

Let us now recall the definition of Schur's algorithm. Take a probability measure $\mu$ on $\T$ with an infinite support and consider the Schur function $f$ of $\mu$. Define the analytic functions $f_n: \D \to \D$ by Schur's algorithm:
\begin{equation}\label{eq41}
z f_{n+1} (z) = \frac{f_{n} (z) - f_{n}(0)}{1 - \ov{f_{n}(0)}f_n (z)}, \quad n \ge 0, \qquad f_0 (z) = f(z), \quad z \in \D , 
\end{equation}
and let $\mu_n$ be the measures generated by $f_n$ via \eqref{eq1}. Note that $f_0 = f$ and $\mu_0 = \mu$. Szeg\H{o} theorem says that the measure $\mu$ is in the Szeg\H{o} class $\szc$ if and only if $\sum_n |f_n (z)|^2 < \infty$ for any $z \in \D$, and, moreover, in this case
\begin{equation}\label{szego}
\int_{\T} \log w \, dm =  \int_{\T} \log (1-|f|^2) \, dm = \sum_{n=0}^\infty \log (1 - |f_n (0)|^2). 
\end{equation}
See Theorem 2.7.14 and Theorem 3.1.4 in Simon \cite{Simonbook1}. In \cite{BD21}, the following generalization of \eqref{szego} was found: 
\begin{equation}\label{entf}
\K(\mu, z) = \sum_{k=0}^{\infty} \log\left(1 + (1 - |z|^2)\frac{|f_{k}(z)|^2}{1 - |f_{k}(z)|^2}\right), \qquad z \in \D,
\end{equation}
for every $\mu \in \szc$. Nonnegativity of summands in this formula implies that the entropy decreases in the Schur algorithm, that is, 
\begin{equation}\label{decayentropy}
\K (\mu_{n+1} , z) \le \K (\mu_n , z) , \qquad z \in \D , \quad  n \ge 0. 
\end{equation}
Let us also mention the useful relation 
\begin{align}\label{entropyschur}
\K(\mu , z) = \log(1- |zf(z)|^2) - \Pp(\log (1- |f|^2), z), \qquad z \in \D, 
\end{align}
that follows from the definition of $\K(\mu, z)$, formula \eqref{eq55} and the fact that $1 - zf$ is an outer function in $\D$ (in particular, $\Pp(\log|1 - \xi f(\xi)|^2, z) = \log|1 - zf(z)|^2$ for every $z \in \D$). The same formula for $\mu_n$, $f_n$ reads as
\begin{align}\label{entropyschur-n}
\K(\mu_n, z) = \log(1- |zf_n(z)|^2) - \Pp(\log (1- |f_n|^2), z), \qquad z \in \D, \quad n \ge 0.
\end{align}
Further, for $\lambda \in \D$, we define
\begin{equation}\label{eq44}
\tilde\phi_{\lambda, n}(z) = \frac{z \phi_{n-1}(z) - \ov{f_{n-1}(\lambda)}\phi_{n-1}^*(z)}{\sqrt{1-|f_{n-1}(\lambda)|^2}}, \quad n \ge 0,
\end{equation}
where we set $f_{-1} \equiv 0$, $z \phi_{-1} \equiv 1$, $\phi_{-1}^* \equiv 1$. Let also $\tilde\phi_{\lambda, n}^* (z) = z^{n}\ov{\tilde\phi_{\lambda, n}(1/\bar z)}$ be the corresponding reflected polynomials. Note that
\begin{equation}\label{eq44star}
\tilde\phi_{\lambda, n}^*(z) = \frac{\phi_{n-1}^*(z) - zf_{n-1}(\lambda)\phi_{n-1}(z)}{\sqrt{1-|f_{n-1}(\lambda)|^2}}, \quad n \ge 0,
\end{equation}
while the usual reflected polynomials $\phi_{n}^*$ can be written in the form
\begin{equation}\notag
\phi_{n}^*(z) = \frac{\phi_{n-1}^*(z) - zf_{n-1}(0)\phi_{n-1}(z)}{\sqrt{1-|f_{n-1}(0)|^2}} = \phi_{0,n}^* (z), \quad n \ge 0,
\end{equation}
as follows from Szeg\H{o} recurrence \eqref{eqreq} and Geronimus theorem, 
$$
a_n = f_n(0) , \quad n \ge 0, 
$$ 
see Theorem 3.1.4 in Simon \cite{Simonbook1}. In particular, for $n\ge 1$ we have
\begin{equation}\label{varphi}
\frac{\tilde\phi_{\lambda, n}^*(z)}{\phi_{n}^*(z)}\frac{\sqrt{1-|f_{n-1}(\lambda)|^2}}{\sqrt{1-|f_{n-1}(0)|^2}}  = \frac{1 - z\ov{f_{n-1}(\lambda)} b_{n-1}(z)}{1 - z\ov{f_{n-1}(0)}b_{n-1}(z)}, \qquad b_{n-1} = \frac{\phi_{n-1}}{\phi^*_{n-1}},
\end{equation}
which yields the estimate 
\begin{equation}\label{eq74}
\frac{1 - |f_{n-1}(\lambda)|}{1 + |f_{n-1}(0)|} \le
\left|\frac{\tilde\phi_{\lambda, n}^*(z)}{\phi_{n}^*(z)}\right|\frac{\sqrt{1-|f_{n-1}(\lambda)|^2}}{\sqrt{1-|f_{n-1}(0)|^2}} \le 
\frac{1 + |f_{n-1}(\lambda)|}{1 - |f_{n-1}(0)|}, 
\end{equation}
that will allow us to switch between $\tilde\phi_{\lambda, n}^*$ and $\phi_{n}^*$ provided that $f_{n-1}(\lambda)$, $f_{n-1}(0)$ are small. Note that \eqref{eq74} trivially holds also for $n = 0$ because $\tilde\phi_{\lambda, 0}^* = \phi_{0}^* = 1$. The main feature of the polynomials $\{\tilde\phi_{\lambda, n}\}$ is the entropy bound
\begin{equation}\label{eq66}
\K(\nu_{\lambda,n}, \lambda) \le \K(\mu, \lambda), \qquad n \ge 0, \quad \nu_{\lambda,n} = \frac{dm}{|\tilde\phi_{\lambda, n}^*|^2},
\end{equation}
see Corollary 4 in \cite{BD21}. Relation \eqref{eq66} could be rewritten in the form
\begin{equation}\label{eq66bis}
\Pp\left(\left|\frac{\tilde\phi_{\lambda, n}^*(\lambda)}{\tilde\phi_{\lambda, n}^*} - 1\right|^2, \lambda \right) \le e^{\K(\mu, \lambda)} - 1,
\end{equation}
see Lemma 2.3 in \cite{Bes24}. The bound \eqref{eq66bis} has important consequences that we state and prove below in Lemma \ref{l50} and Lemma \ref{lemNearCircAdd}. These lemmas are not specific for orthogonal polynomials as they work for any polynomials that satisfy entropy bounds of type \eqref{eq66bis} with the small enough right hand side. For consistency, we will use notation $\phi^*$ for a polynomial without zeroes in the open unit disk. The lemmas will be used later for $\phi^* = \tilde\phi_{\lambda, n}^*$. 

\medskip

\begin{Lem}\label{l50}
Let $\phi^*$ be a polynomial of degree at most $n$ without zeroes in the open unit disk $\D$, and let $\phi (z) = z^{n}\ov{\phi^*(1/\bar z)}$ be the corresponding reflected polynomial. Denote by $b = \phi/\phi^*$ the Blaschke product generated by $\phi$, $\phi^*$. Fix $\zeta \in \T$, $r \in (0, 1)$, $A > 0$, and consider the Stolz angle $\Gamma_{\zeta} = \conv\bigl(\{z \in \C:\; |z| \le r\},\, \{\zeta\}\bigr)$. For every $\eps > 0$ there exists $\eta_0 > 0$ depending only on $r$, $A$, $\eps$, such that if 
\begin{equation}\label{eq66bisbis}
\Pp\left(\left|\frac{\phi^*(\lambda)}{\phi^*} - 1\right|^2, \lambda \right) \le \eta_0
\end{equation}
for some $\lambda \in \Gamma_{\zeta}$, then there exists $\alpha(\lambda) \in \T$ such that 
\begin{equation}\label{eq774}
\sup_{z:\; d_{H}(z,\lambda) \le A}\bigl|b(z) - \alpha(\lambda) z^n\bigr| \le \eps, 
\end{equation}
where $d_{H}(z,\lambda)$ is the hyperbolic distance between $z$, $\lambda$, see \eqref{eqHD}.
\end{Lem}
\beginpf Put $E'_{\eps}(\lambda)= \{\xi \in \T: |\phi^*(\lambda)/\phi^*(\xi) - 1| > \eps\}$, $E_{\eps}(\lambda) = \T\setminus E'_{\eps}(\lambda)$, and
$$
\eta = \Pp\left(\left|\frac{\phi^*(\lambda)}{\phi^*} - 1\right|^2, \lambda \right).
$$
We have
\begin{equation}\label{eq121}
\sup_{z:\; d_{H}(z,\lambda) \le A}\int_{E'_{\eps}(\lambda)}\frac{1 - |z|^2}{|1 -\bar\xi z|^{2}} \,dm(\xi) \to 0, \qquad \eta \to 0.
\end{equation}
Since on the set $E_{\eps}(\lambda)$ the argument of $\phi^*$ (to be denoted by $\arg\phi^*$) is close to the argument of $\phi^*(\lambda)$ modulo $2\pi \Z$, and since $b(\xi) = \xi^n \cdot e^{-2i\arg\phi^*(\xi)}$ for all $\xi \in \T$, we have 
$$
\sup_{\xi \in E_{\eps}(\lambda)}|b(\xi) - \alpha(\lambda)\xi^n| \lesssim \eps, \qquad \alpha(\lambda) = \frac{\ov{\phi^*(\lambda)}}{\phi^*(\lambda)} = e^{-2i\arg\phi^*(\lambda)}.
$$  
We now see from \eqref{eq121} that
$$
\sup_{z:\; d_{H}(z,\lambda) \le A}\int_{\T}|b(\xi) - \alpha(\lambda)\xi^n|\frac{1 - |z|^2}{|1 -\bar\xi z|^{2}} \,dm(\xi) \lesssim \eps
$$ 
if $\eta$ is small enough. This relation and the analyticity of  $b$, $z^n$ imply \eqref{eq774}  provided that \eqref{eq66bisbis} holds with $\eta_0 > 0$ small enough. \qed

\medskip

Next lemma was implicitly used in \cite{Bes24} but not proved there. We fill this gap below.

\begin{Lem}\label{lemNearCircAdd}
Let $\phi^*$ be a polynomial of degree at most $n$ without zeroes in the open unit disk $\D$, let $\zeta \in \T$, $A_1, A_2 > 0$. Assume that $\max(A_1, A_2)/n \le 1/2$ and set $\lambda_n = (1-A_1/n)\zeta$. There exists $\eta_0 > 0$ depending only on $A_1$, $A_2$ such that if  
\begin{equation}\label{eq66bisbisbis}
\Pp\left(\left|\frac{\phi^*(\lambda_n)}{\phi^*} - 1\right|^2, \lambda_n \right) \le \eta_0,
\end{equation}
then the polynomial $\phi^*$ has no zeroes in $B(\zeta, A_2/n) = \{z\in \C:\; |\zeta - z| \le A_2/n\}$. 
\end{Lem}
\beginpf Define 
$$
\eta = \Pp\left(\left|\frac{\phi^*({\lambda_n})}{\phi^*} - 1\right|^2, {\lambda_n} \right).
$$
Our first goal is to show that if $\eta$ is small enough, then $\phi^*$ cannot have zeroes in the ``$\eps$-strip'' 
$$
\{z \in B(\zeta, A_2/n): \; \bigl|1- |z|\bigr| < \eps/n\},
$$
for some $\eps > 0$. Suppose, on the contrary, that for a given $\eps \in (0, A_2/2)$ there is $z^* \in B(\zeta, A_2/n)$ such that $|1 - |z^*| | < \eps/n$ and  $\phi^*(z^*) = 0$. Since $\phi^*$ has no zeroes in $\D$, we have $|z^*| \ge 1$. Choose a wide Stolz angle $\Gamma = \Gamma_{\zeta}^{A_2, \eps}$ with vertex at $\zeta$ and aperture sufficiently large such that the Hausdorff distance between $\Gamma \cap B(\zeta, A_2/n)$ and the unit circle  $\T$ is comparable to $\eps/n$ (see Figure \ref{fig:img3}). Observe that the point $z_{{\lambda_n}} = \left(1 - \frac{A_2}{n}\right)\frac{z^*}{|z^*|}$ belongs to $\Gamma$. 
Denote by $z_\eps$ the point on the boundary of  $\Gamma$ such that $z_\eps/|z_\eps| = z^*/|z^*|$. Since all three points $\lambda_n$, $z_\eps$, and $z_{\lambda_n}$ belong to the same Stolz angle $\Gamma$ and have comparable distances to the unit circle $\T$, we have
\begin{align*}
\left|\frac{\phi^*({\lambda_n})}{\phi^*(z_{{\lambda_n}})} - 1\right| 
& \le \sqrt{\Pp\left(\left|\frac{\phi^*({\lambda_n})}{\phi^*} - 1\right|^2, z_{{\lambda_n}} \right)}\\
&\lesssim \sqrt{\Pp\left(\left|\frac{\phi^*({\lambda_n})}{\phi^*} - 1\right|^2, {\lambda_n}\right)} \lesssim \sqrt{\eta},\\
\left|\frac{\phi^*({\lambda_n})}{\phi^*(z_\eps)} - 1\right| 
&\le \sqrt{\Pp\left(\left|\frac{\phi^*({\lambda_n})}{\phi^*} - 1\right|^2, z_\eps \right)}\\
&\lesssim \sqrt{\Pp\left(\left|\frac{\phi^*({\lambda_n})}{\phi^*} - 1\right|^2, {\lambda_n} \right)} \lesssim \sqrt{\eta},
\end{align*} 
with constants depending only on  $A_1$, $A_2$ and $\eps$. This gives 
\begin{equation}\label{eq92}
\left|\frac{\phi^*(z_{{\lambda_n}})}{\phi^*(z_\eps)}\right| = 1 + O(\sqrt{\eta}) 
\end{equation}
for all $\eta < \eta_{00}$ (again, the value of $\eta_{00}$ depends solely on $A_1$, $A_2$ and $\eps$). 
\begin{figure}
\centering
  \includegraphics[width=\linewidth]{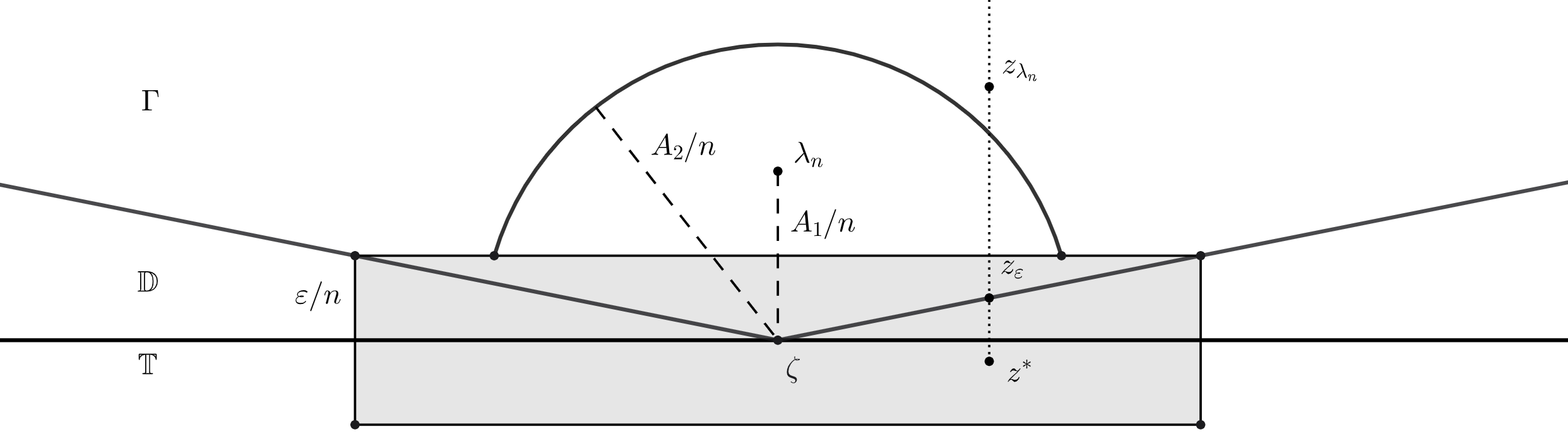}
\caption{Objects appearing in the proof of Lemma \ref{lemNearCircAdd}.}
\label{fig:img3}
\end{figure}
At the same time, Lemma 2.3 in \cite{B21} implies
the estimate
$$
\left|\frac{\phi^*(z_{{\lambda_n}})}{\phi^*(z_\eps)}\right| \ge \left(\frac{1+r_\eps}{2}\right)^{n-1}\left|\frac{z_{{\lambda_n}}-z^*}{z_\eps-z^*}\right|, \qquad r_\eps = \left|\frac{z_{{\lambda_n}}}{z_\eps}\right|, 
$$
if we factorize $\phi^* = (z-z^*)p_{n-1}$ and use the fact that the polynomial $p_{n-1}$ has no zeroes in $\D$ (the first usage of such estimates in the context of orthogonal polynomials is due to M\'at\'e and Nevai \cite{MN80}). Observe that 
$$
\left(\frac{1+r_\eps}{2}\right)^{n-1} \ge \left(\frac{1+|z_{{\lambda_n}}|}{2}\right)^{n-1} \ge \left(1 - \frac{A_2}{2n}\right)^{n-1} \ge c_{A_2},
$$ 
for some constant $c_{A_2} > 0$, and 
$$
\left|\frac{z_{{\lambda_n}}-z^*}{z_\eps-z^*}\right| \ge \frac{A_2}{\eps}.
$$ 
So, if $\eps > 0$ is so small that 
$$
\left(\frac{1+r_\eps}{2}\right)^{n-1}\left|\frac{z_{\lambda_n}-z^*}{z_\eps-z^*}\right| \ge \frac{A_2c_{A_2}}{\eps} \ge 2,
$$
and $\eta_{00} > 0$ is so small that for $\eta \le \eta_{00}$ relation \eqref{eq92} gives 
$$
\left|\frac{\phi^*(z_{\lambda_n})}{\phi^*(z_\eps)}\right| \le \frac{3}{2},  
$$
we obtain a contradiction. In other words, for chosen $\eps$, $\eta_{00}$ there is no point $z^* \in B(\zeta, A_2/n)$ satisfying assumptions 
\begin{equation}\notag
|1 - |z^*| | \le \eps/n, \qquad \phi^*(z^*) = 0, 
\end{equation}
provided $\eta \le \eta_{00}$. Our next aim is to prove that for this fixed value $\eps>0$ the polynomial $\phi^*$ has no zeroes also in the remaining part of the disk $B(\zeta, A_2/n)$ (in fact, for this we will need to change $\eta_{00}$ by a smaller constant). 

Set $\phi (z) = z^n \ov{\phi^*(1/\bar z)}$ and note that
$b = \phi/\phi^*$ is a Blaschke product. Lemma \ref{l50} implies that for every $A > 0$, $\gamma > 0$ there exists $\eta_{01} > 0$ such that if $\eta \le \eta_{01}$, then  
\begin{equation}\label{eq94}
\sup_{z:\; d_{H}(z,\lambda) \le A}\Bigl||b(z)| - |z|^n\Bigr| \le \gamma. 
\end{equation}
For $S \subset \D$, we will denote $S^* = \{z \in \C: 1/\bar z \in S\}$. Let us take $A$ so large that the hyperbolic disk $\Omega_n(A) = \{z \in \D:\; d_{H}(z,\lambda_n) \le A\}$ satisfies
\begin{equation}\label{eq114}
\Omega_n(A) \cup \{z \in \C: \; \bigl|1- |z|\bigr| < \eps/n\} \cup \Omega_n(A)^* \supset B(\zeta, A_2/n).
\end{equation}
Note that the choice of $A$ depends only on parameters $\eps$, $A_2$ and can be made independent of $n$. In fact, we have 
$$
\inf_{z \in \Omega_n(A)}(1-|z|) \asymp 1/n,
$$ 
which guarantees that $\gamma$ defined by
$$
\gamma = \frac{\inf_{z \in \Omega_n(A)}|z|^n}{2}
$$ 
is bounded below by a constant only depending in $A$. Then, for this choice of $A$, $\gamma$ we use Lemma \ref{l50} to find $\eta_{01}>0$ such that \eqref{eq94} holds for $\eta \le \eta_{01}$. Note that \eqref{eq94} implies that the Blaschke product $b$ has no zeroes in $\Omega_n(A)$. Hence, the same is true about $\phi$, and then the polynomial $\phi^*$ has no zeroes in the reflected set $\Omega_n(A)^*$. Finally, \eqref{eq114} and the first part of the proof imply that $\phi^*$ has no zeroes in the whole disk $B(\zeta, A_2/n)$. We now see that the claim of the lemma holds for $\eta_{0} = \min(\eta_{00}, \eta_{01})$. \qed

\medskip

The following result is Lemma 2.6 in \cite{Bes24}.
\begin{Lem}\label{lemNearCirc}
Let $\mu \in \szc$, $A > 0$, $\zeta \in \T$. Assume that $n \ge 2A$ and set $\lambda = (1-A/2n)\zeta$. There exists a constant $\eta_0 > 0$ depending only on $A$, such that if $\K(\nu_{\lambda,n}, \lambda)  \le \eta_0$, then 
$$
\left|\frac{\tilde\phi_{\lambda, n}^*(\lambda)}{\tilde\phi_{\lambda, n}^*(z)} - 1\right| \le c e^{2A}\sqrt{\K(\nu_{\lambda,n}, \lambda)},
$$
for every $z \in \partial \Omega_n$, where $\Omega_n$ is a domain in $\C$ with a piece-wise smooth boundary $\partial \Omega_n$ such that $B(\zeta, A/n) \subset \Omega_n \subset B(\zeta, 4A/n)$, $B(\zeta, r) = \{z\in \C:\; |\zeta - z| \le r\}$. Here, the constant $c$ is universal.
\end{Lem}

Let $\eta_0$ be the best possible $\eta_0$ in Lemma \ref{lemNearCircAdd}, Lemma \ref{lemNearCirc} for the choice of parameters $A_1 = 2$, $A_2 = 8$, $A=2$. In other words, let $\eta_0$ be the maximum of the all numbers $\eta_0$ for which {\it both} Lemma \ref{lemNearCircAdd} and Lemma \ref{lemNearCirc} work for $A_1 = 2$, $A_2 = 8$, $A=2$. Take a point $\zeta \in \T$ such that \eqref{eq21} holds. Set $\lambda_j = (1- 1/j)\zeta$ and define
\begin{equation}\label{nova}
n_0 = \min\Bigl\{n \ge 8A: \; e^{\K(\mu, \lambda_j)} - 1 \le \eta_0 \mbox{ for all }j \ge n\Bigr\}.
\end{equation}
We arrive at the following estimates.
\begin{Lem}\label{domain-bound}
For every $n \ge n_0$ and $0 \le k \le n$, we have
\begin{equation}\label{eq115bis}
\left|\frac{\tilde\phi_{\lambda_n, k}^*(z)}{\tilde\phi_{\lambda_n, k}^*(\lambda_n)} - 1 \right| \lesssim \sqrt{\K(\mu, \lambda_n)}, \qquad z \in B(\zeta, 1/n).
\end{equation}
Consequently, we have
\begin{equation}\label{eq68}
\left|\frac{\tilde\phi_{\lambda_n, k}^*(z)}{\tilde\phi_{\lambda_n, k}^*(\lambda_n)}\right|^2 = 1+O(\sqrt{\K(\mu, \lambda_n)}), \qquad z \in B(\zeta, 1/n), 
\end{equation}
where the constant in $O(\cdot)$ does not depend on $\mu$, $k$, and $z$.
\end{Lem}
\beginpf Take $n \ge n_0$, $0 \le k \le n$. Let us use Lemma \ref{lemNearCircAdd} for $\phi^* = \tilde\phi_{\lambda_n, k}^*$ with the parameters $\tilde n = k$, $\tilde A_1 = k/n$, $\tilde A_2 = 8k/n$ in place of $n$, $A_1$, $A_2$. We have 
$\max(\tilde A_1, \tilde A_2)/\tilde n = 8/n \le 8/n_0 \le 1/2$, as required in Lemma \ref{lemNearCircAdd}. We also have
$$
\lambda_{\tilde n} = (1 - \tilde A_1/\tilde n)\zeta = (1 - 1/n)\zeta = \lambda_{n}.
$$
Note that $\tilde A_1 \le 1$, $\tilde A_2 \le 8$, therefore, the optimal (i.e., largest possible) number $\tilde\eta_0$ in Lemma \ref{lemNearCircAdd} for this choice of parameters satisfies $\eta_0 \le \tilde \eta_0$. Let us check the main assumption \eqref{eq66bisbisbis} in Lemma \ref{lemNearCircAdd}:
\begin{equation}
\Pp\left(\left|\frac{\phi^*(\lambda_n)}{\phi^*} - 1\right|^2, \lambda_n \right) =  
e^{\K(\nu_{\lambda_n, k}, \lambda_n)} - 1 \le
e^{\K(\mu, \lambda_n)} - 1 \le \eta_0 \le \tilde\eta_0.
\end{equation}
Lemma \ref{lemNearCircAdd} now tells us that $\tilde\phi_{\lambda_n, k}^*$ has no zeroes in $B(\zeta, \tilde A_2/\tilde n) = B(\zeta, 8/n)$. This information will be used at the last part of the proof.

\medskip

Next, we use Lemma \ref{lemNearCirc} for $\tilde n = k$, $\tilde A = 2k/n$ in place of $n$, $A$. With this choice of parameters, we have
$\tilde \lambda = (1 - \tilde A/2\tilde n)\zeta = \lambda_n$. Moreover, 
$$
\K(\nu_{\tilde \lambda,\tilde n}, \tilde \lambda) \le \K(\mu, \tilde\lambda) = \K(\mu, \lambda_n) \le e^{\K(\mu, \lambda_n)} - 1 \le \eta_0 \le \tilde\eta_0, 
$$
for $n \ge n_0$. Here $\tilde\eta_0$ is the largest possible number in Lemma \ref{lemNearCirc} for $\tilde n$, $\tilde A$, and we have used the fact that $\tilde A \le 2$, so $\eta_0 \le \tilde\eta_0$. We see that 
\begin{equation}\label{eq117}
\left|\frac{\tilde\phi_{\lambda_n, k}^*(\lambda_n)}{\tilde\phi_{\lambda_n, k}^*(z)} - 1\right| = \left|\frac{\tilde\phi_{\tilde\lambda, \tilde n}^*(\tilde \lambda)}{\tilde\phi_{\tilde \lambda, \tilde n}^*(z)} - 1\right| \lesssim e^{2\tilde A}\sqrt{\K(\nu_{\tilde \lambda, \tilde n}, \tilde \lambda)} \lesssim \sqrt{\K(\mu, \lambda_n)},
\end{equation}
for every $z \in \partial \Omega_{\tilde n}$, where $\Omega_{\tilde n}$ is a domain in $\C$ with a piece-wise smooth boundary $\partial \Omega_{\tilde n}$ such that $B(\zeta, 2/n) \subset \Omega_{\tilde n} \subset B(\zeta, 8/n)$. In the last chain of inclusions we have used  the fact that $\tilde A/\tilde n = 2/n$, $4\tilde A/\tilde n = 8/n$. Then from the maximum modulus principle for the function $1/\tilde\phi^*_{\lambda_n, k}$ in the domain $\Omega_{\tilde n}$, we obtain
\begin{equation}\label{eq115}
\left|\frac{\tilde\phi_{\lambda_n, k}^*(\lambda_n)}{\tilde\phi_{\lambda_n, k}^*(z)} - 1\right| \lesssim \sqrt{\K(\mu, \lambda_n)},
\qquad z \in \Omega_{\tilde n}.
\end{equation}
In other words, \eqref{eq117} holds everywhere in $\Omega_{\tilde n}$, not just at the boundary $\partial \Omega_{\tilde n}$. We used here the fact that $\tilde\phi_{\lambda_n, k}^*$ has no zeroes in $B(\zeta, 8/n)$, hence the function $1/\tilde\phi^*_{\lambda_n, k}$ is analytic in $\Omega_{\tilde n} \subset B(\zeta, 8/n)$. Since we also have $\Omega_{\tilde n} \supset B(\zeta, 2/n)$, relations \eqref{eq115bis}, \eqref{eq68} follow from \eqref{eq115}. \qed

\medskip

The bounds \eqref{eq115bis}, \eqref{eq68} will be among the main ingredients in the proofs of Theorem \ref{thm1} and Theorem \ref{thm2}.

\medskip

\section{Proof of Theorem \ref{thm1}}
\noindent Define the nonnegative functions $\eta_n$, $\etat_n$ in the open unit disk $\D$ by
\begin{align}
\eta_n(z) &= \frac{1}{n}\sum_{k=0}^{n-1} \K(\mu_k, z), \label{etan}\\
\etat_n(z) &= \frac{1}{n}\sum_{k=0}^{n-1}\log\frac{1}{1 - |f_k(z)|}. \label{etatn}
\end{align}
\begin{Lem}\label{l31}
Let $\mu \in \szc$. We have 
$\sup_{z \in \Gamma_\zeta} \eta_n(z) \to 0$ as $n \to \infty$ for every $\zeta \in \T$ such that  \eqref{eq21} holds.
\end{Lem}
\beginpf The estimate \eqref{decayentropy} gives $\eta_n (z) \le \K(\mu , z)$ for any $z \in \D$ and $n \ge 1$. Hence, it is sufficient to prove that $\sup\{\K(\mu_n , z): |z| < r \} \to 0$ as $n \to \infty$, for any $0<r<1$. Observe that \eqref{entropyschur} gives 
$$
\K(\mu_n , z) \le - \Pp(\log (1- |f_n|^2), z) \le -\frac{1+|z|}{1-|z|}\int_{\T}\log(1-|f_n|^2)\,dm
$$
and one only needs to check that the last integral tends to $0$ as $n \to \infty$. In turn, this follows from Szeg\H{o} formula \eqref{szego} for the pair $\mu_n$, $f_n$, i.e., from
\begin{equation}\label{eq57}
\int_{\T}\log(1-|f_n|^2)\,dm = \sum_{k=n}^\infty \log (1 - |f_k(0)|^2).
\end{equation}
Since  by  \eqref{szego} we have $\sum_{k\ge 0} |f_k(0)|^2 < \infty$ for every $\mu\in\szc$, the right hand side in \eqref{eq57} indeed tends to $0$ as $n \to \infty$. \qed

\medskip

\begin{Lem}\label{l32}
Let $\mu \in \szc$, and let
$$
\Gamma_{\zeta, n} = \Gamma_{\zeta} \cap \{z \in \C:\, |z - \zeta| \ge 1/n\}, \qquad n \ge 1.
$$
We have $\sup_{z \in \Gamma_{\zeta, n}} \etat_n(z) \to 0$ as $n \to \infty$ for every $\zeta \in \T$ such that  \eqref{eq21} holds.
\end{Lem}
\beginpf We first deal with $z \in  \Gamma_{\zeta, n}$ such that $\K(\mu, z) \le \log 2$. Let us divide the set of indices $I_n =\{k: 0 \le k \le n-1\}$ into two subsets 
\begin{align}
I_{n, s} 
&= \left\{k \in I_n:\;\;  |f_k(z)| \le 1/2\right\}, \label{eqIns}\\ 
I_{n, b} 
&= \left\{k \in I_n:\;\;  |f_k(z)| > 1/2\right\}. \label{eqInb}
\end{align}
For $k \in I_{n, s}$, we have
\begin{align}
\frac{1}{n}\sum_{k \in I_{n,s}}\log\frac{1}{1 - |f_k(z)|} 
&\lesssim \frac{1}{n} \sum_{k=0}^{n-1} |f_k(z)| \le \sqrt{\frac{1}{n} \sum_{k=0}^{n-1} |f_k(z)|^2}\notag\\
& \lesssim \sqrt{(1-|z|^2)\sum_{k=0}^{n} |f_k(z)|^2} \lesssim \sqrt{\K(\mu, z)},\label{eq51}
\end{align}
where in the last estimate we have used \eqref{entf} and the fact that $x \lesssim \log (1+ x)$ for $x \in [0, 1]$. Next, for $k \in I_{n, b}$ we have
$$
\frac{1}{n}\sum_{k \in I_{n,b}}\log\frac{1}{1 - |f_k(z)|} \le \frac{1}{n}\sum_{k \in I_{n,b}}\frac{1 + |f_k(z)|}{1 - |f_k(z)|^2} \lesssim \sum_{k \in I_{n,b}}\frac{(1-|z|^2)|f_k(z)|^2}{1 - |f_k(z)|^2}.
$$
Each summand in the last sum satisfies  
$$
0 \le \frac{(1-|z|^2)|f_k(z)|^2}{1 - |f_k(z)|^2} \le 1,
$$
for otherwise we cannot have $\K(\mu, z) \le \log 2$, see \eqref{entf}. Thus, one can use again the elementary inequality $x \lesssim \log(1+x)$ for $x \in [0, 1]$ and obtain
\begin{equation}\label{eq52}
\frac{1}{n}\sum_{k \in I_{n,b}}\log\frac{1}{1 - |f_k(z)|} \lesssim \sum_{k \in I_{n,b}}\log \left(1 +\frac{(1-|z|^2)|f_k(z)|^2}{1 - |f_k(z)|^2}\right) \le \K(\mu, z).
\end{equation}
Combining \eqref{eq51}, \eqref{eq52}, we arrive at
\begin{align}
\etat_n(z) 
&= \frac{1}{n}\sum_{k \in I_{n,s}}\log\frac{1}{1 - |f_k(z)|} + \frac{1}{n}\sum_{k \in I_{n,b}}\log\frac{1}{1 - |f_k(z)|}\notag \\
&\lesssim \sqrt{\K(\mu, z)} + \K(\mu, z) \lesssim \sqrt{\K(\mu, z)},\label{eq53}
\end{align}
for every $z \in  \Gamma_{\zeta, n}$ such that $\K(\mu, z) \le \log 2$. Moreover, the constants in \eqref{eq53} are universal. Thus, the result will follow from \eqref{eq21} if we show that $\etat_n$ tends to zero uniformly on compact subsets of $\D$. In turn, this will follow if we show that 
$\max_{|z|\le r}|f_n(z)| \to 0$ as $n\to \infty$ for every $r \in (0, 1)$. So, let us take $z$ with $|z|\le r< 1$ and estimate
\begin{multline*}
|f_n(z)|^2 \le \log\frac{1}{1 - |f_n(z)|^2} = -\log\left(1 - |\Pp(f_n, z)|^2\right) \le \\
\le -\Pp\left(\log(1-|f_n|^2), z\right) \lesssim -\frac{1+r}{1-r}\int_{\T} \log (1-|f_n|^2)\,dm,
\end{multline*}
where we used Jensen's inequality for the convex function $x \mapsto -\log(1-x^2)$ and the basic bound for the Poisson kernel. Then, as in the proof of Lemma \ref{l31},
$$
\lim_{n \to \infty}\int_{\T} \log (1-|f_n|^2)\,dm = \lim_{n \to \infty}\sum_{k=n}^\infty \log (1 - |f_k (0)|^2) = 0,
$$
and the result follows. \qed

\medskip

\begin{Lem}\label{l1}
Let $\mu = w\,dm +\mus$ be a measure in $\szc$, let $D_\mu$ be the corresponding Szeg\H{o} function, and let $\{\phi_n\}_{n \ge 0}$ be the sequence of orthonormal polynomials generated by $\mu$, see \eqref{eqonb}. We have 
\begin{equation}\label{eqKH}
\phi_n^* D_{\mu} = \frac{O_{n}}{1 - zb_n f_n}, \qquad b_n = \frac{\phi_n}{\phi_n^*}, \qquad n \ge 0,
\end{equation}
for the outer function $O_{n}$ defined by the conditions $O_{n}(0) > 0$ and $|O_{n}|^2 = 1-|f_n|^2$ almost everywhere on $\T$. 
\end{Lem}
\beginpf Recall that the reflected polynomials $\phi_n^*$ have no zeroes in the closed unit disk, see, e.g., Theorem 1.7.1 in Simon \cite{Simonbook1}. Therefore, $b_n$ is a finite Blaschke product, $\Re(1 - zb_n f_n) \ge 0$ in $\D$, and the function $1 - zb_n f_n$ is outer (see Corollary 4.8 in Section II.4 in Garnett \cite{Garnett}). From here we see that functions $\phi_n^* D_{\mu}$ and $O_{n} (1 - zb_n f_n)^{-1}$ are outer as well. Note that  they have positive value at $0$. Thus, we only need to check that 
\begin{equation}\label{eq61}
|\phi_n^* D_{\mu}| = \left|\frac{O_{n}}{1 - zb_n f_n}\right|,
\end{equation}
almost everywhere on $\T$. Taking the square and using the definitions of $D_\mu$, $O_n$, we see that \eqref{eq61} is equivalent to 
$$
|\phi_n^*|^2 w = \frac{1-|f_n|^2}{|1 - zb_n f_n|^2}
$$
almost everywhere on $\T$. This formula is due to S.\,Khrushchev, see Theorem 2 on page 173 in \cite{Kh01}. \qed

\medskip

\begin{Lem}\label{lemAv}
We have 
\begin{align}
\frac{1}{n}\sum_{k=0}^{n-1} \Bigl| \log(1 - z b_k(z) f_k(z)) \Bigr| \le \etat_n(z),& \label{eq42}\\
\frac{1}{n}\sum_{k=0}^{n-1}\log|O_k(z)|^{-2} \le \eta_n(z) + \etat_n(z),& \label{eq43}
\end{align}
for every $z \in \D$. In \eqref{eq42}, we deal with the main branch of the logarithm, $\log 1 = 0$. 
\end{Lem}
\beginpf Inequality \eqref{eq42} is immediate from the definition of $\etat_n$ and the estimate
$$
\Bigl| \log(1 - z b_k(z) f_k(z)) \Bigr| \le \log\frac{1}{1 - |z b_k(z) f_k(z)|} \le  \log\frac{1}{1 - |f_k(z)|}, \qquad z \in \D.
$$
To prove \eqref{eq43}, we write 
\begin{align*}
\frac{1}{n}\sum_{k=0}^{n-1}\log|O_k(z)|^{-2}
&= \frac{1}{n}\sum_{k=0}^{n-1}\Pp(\log|O_k|^{-2}, z)\\
&= -\frac{1}{n}\sum_{k=0}^{n-1}\Pp(\log(1 - |f_k|^2), z)\\
&= \frac{1}{n}\sum_{k=0}^{n-1}\K(\mu_k, z) - \frac{1}{n}\sum_{k=0}^{n-1}\log (1 - |zf_k(z)|^2)\\
&\le \eta_n(z) + \etat_n(z).
\end{align*}
First equality here holds because $O_k$ is an outer function for each $k \ge 0$. In the second line we used the fact that that $|O_k|^2 = 1 - |f_k|^2 $ on $\T$, $k \ge 0$. Equality in the third line follows from \eqref{entropyschur-n}, and the last estimate from $-\log (1 - |zf_k(z)|^2) \le -\log (1 - |f_k(z)|)$. \qed

\medskip

\noindent Below we will use the bound
\begin{equation}\label{eq45}
\frac{1}{n}\sum_{k=0}^{n-1}|D_\mu(z)\phi_{k}^*(z)|^2 \le 1 + \delta_n(\mu,z), \qquad z \in \D,
\end{equation}
where 
$$
\delta_n(\mu,z) = 2\sqrt{\frac{e^{\K(\mu, z)} - 1}{n(1-|z|^2)}} + 4 \frac{e^{\K(\mu, z)} - 1}{n(1-|z|^2)}.
$$
For the proof of \eqref{eq45}, see Lemma 2.2 in \cite{B21}.
\begin{Lem}\label{lemSt-n}
We have
$$
\frac{1}{n}\sum_{k=0}^{n-1}\Bigl||\phi_k^*(z) D_{\mu}(z)|^2 - 1\Bigr| \lesssim \delta_n(\mu,z) + \eta_n(z) + \etat_n(z), 
$$
for every $z \in \D$. 
\end{Lem}
\beginpf Fix $z \in \D$. Let us divide the set of indices $J_n =\{k: 0 \le k \le n-1\}$ into two subsets, 
\begin{align}
J_{n, s} 
&= \left\{k \in J_n:\;\;  \Bigl|\log|\phi_k^*(z) D_{\mu}(z)|\Bigr| \le 1\right\}, \label{eq71}\\ 
J_{n, b} 
&= \left\{k \in J_n:\;\;  \Bigl|\log|\phi_k^*(z) D_{\mu}(z)|\Bigr| > 1\right\}. \label{eq72}
\end{align}
For $k \in J_{n, s}$, we have
$$
\Bigl||\phi_k^*(z) D_{\mu}(z)|^2 - 1\Bigr| \lesssim \Bigl|\log|\phi_k^*(z) D_{\mu}(z)|^2\Bigr| \le 
\Bigl|\log|O_k(z)|^2\Bigr| + \Bigl|\log|1 - z b_k(z) f_k(z)|^2\Bigr|,
$$
where in the second inequality we have used Lemma \ref{l1}. From \eqref{eq42}, \eqref{eq43} we obtain 
\begin{equation}\label{eq63}
\frac{1}{n}\sum_{k \in J_{n,s}}\Bigl||\phi_k^*(z) D_{\mu}(z)|^2 - 1\Bigr| \lesssim \eta_n(z) + \etat_n(z), \qquad z \in \D.
\end{equation}
To handle indices $k \in J_{n,b}$, we use \eqref{eq45}. Together with \eqref{eq63} it gives 
\begin{align}
\frac{1}{n}\sum_{k \in J_{n,b}}|\phi_k^*(z) D_{\mu}(z)|^2 
&=\frac{1}{n}\sum_{k=0}^{n-1}|D_\mu(z)\phi_{k}^*(z)|^2 - \frac{1}{n}\sum_{k \in J_{n,s}}|\phi_k^*(z) D_{\mu}(z)|^2\notag\\
&\le 1 + \delta_n(\mu,z) - \frac{|J_{n,s}|}{n+1} + O(\eta_n(z) + \etat_n(z))\notag\\
&= \frac{|J_{n,b}|}{n+1} + \delta_n(\mu,z) + O(\eta_n(z) + \etat_n(z)).\label{eq46}
\end{align}
It follows that
\begin{equation}\label{eq77}
\frac{1}{n}\sum_{k \in J_{n,b}}\Bigl||\phi_k^*(z) D_{\mu}(z)|^2 - 1\Bigr| \le 2\frac{|J_{n,b}|}{n} + \delta_n(\mu,z) + O( \eta_n(z) + \etat_n(z)).
\end{equation}
Recall that the sets $I_{n,s}$, $I_{n,b}$ were defined in \eqref{eqIns}, \eqref{eqInb}, respectively. We have
\begin{equation*}
\frac{|J_{n,b}|}{n} = \frac{|J_{n,b} \cap I_{n, s}|}{n} + \frac{|J_{n,b} \cap I_{n, b}|}{n} \le \frac{|J_{n,b} \cap I_{n, s}|}{n} + \frac{|I_{n, b}|}{n}. 
\end{equation*}
Note that 
$$
\etat_n (z) = \frac{1}{n} \sum_{k=0}^{n-1} \log \frac{1}{1-|f_k (z)|} \geq \frac{|I_{n,b}|}{n} \log 2, 
$$
and if $k \in J_{n,b} \cap I_{n,s}$, Lemma \ref{l1} gives $\log |O_k (z)|^{-2} \gtrsim 1$. Then \eqref{eq43} yields
$$
\frac{|J_{n,b} \cap I_{n,s}|}{n} \lesssim \frac{1}{n}\sum_{k=0}^{n-1}\log|O_k(z)|^{-2} \le \eta_n(z) + \etat_n(z). 
$$
We deduce that 
\begin{equation}\label{eq47}
\frac{|J_{n,b}|}{n} \lesssim \eta_n(z) + \etat_n(z) . 
\end{equation}
Taking into account \eqref{eq77}, this completes the proof. \qed

\medskip

Lemma \ref{lemSt-n} together with Lemma \ref{l31} and Lemma \ref{l32} will give us the asymptotic relation 
\begin{equation}\label{eq101}
\sup_{z \in \Gamma_{\zeta, n}}\frac{1}{n}\sum_{k=0}^{n-1}\Bigl||\phi_k^*(z) D_{\mu}(z)|^2 - 1\Bigr| \to 0, \qquad n\to \infty.
\end{equation}
In the remaining piece $\Gamma'_{\zeta, n} := \Gamma_{\zeta}\setminus \Gamma_{\zeta, n}$ of the Stolz angle $\Gamma_\zeta$ we will need a different argument. For convenience, let us reproduce the definition of $\Gamma'_{\zeta, n}$ in the following form:
$$
\Gamma'_{\zeta, n} = \Gamma_{\zeta} \cap \{z \in \C:\, |z - \zeta| < 1/n\}, \qquad n \ge 1.
$$
We also define
\begin{align}
R_1(n) &= \sup_{z_{1,2} \in \Gamma'_{\zeta,n}}\left|\frac{D_{\mu}(z_1)}{D_{\mu}(z_2)}\right|^2,\label{eq69}\\
R_2(n) &= \sup_{z \in \Gamma_{\zeta,n}}(\delta_n(\mu,z) + \eta_n(z) + \etat_n(z)).
\end{align}
It is clear that $R_1(n) \to 1$ and $R_2(n) \to 0$ for every $\zeta \in \T$ such that \eqref{eq21} holds and $|D_\mu|$ has a non-zero finite non-tangential limit at $\zeta$.
\begin{Lem}\label{l36}
Let $\zeta \in \T$. For every $\Lambda \subset\{0\le k \le n-1\}$ we have
\begin{equation}\label{eq73}
\sup_{z \in \Gamma'_{\zeta,n}}\frac{\sum_{k \in \Lambda}\Bigl||\phi_k^*(z) D_{\mu}(z)|^2 - 1\Bigr|}{n}  \lesssim\frac{R_1(n)|\Lambda|}{n} + R_1 (n) R_2(n),
\end{equation}
for all $n \ge 1$.
\end{Lem}
\beginpf Take $z \in \Gamma'_{\zeta,n}$ and consider $\lambda(z) = rz$ where $r \in (1-1/n, 1)$ is chosen so that $|\lambda(z)| = 1-1/n$. Note that $\lambda(z) \in \Gamma_{\zeta,n}$. Since $\phi_k^*$ has no zeroes in $\D$, we have
$$
|\phi_k^*(\lambda(z))|^2 \ge \left(\frac{1+r}{2}\right)^k|\phi_k^*(z)|^2  \gtrsim |\phi_k^*(z)|^2, \qquad 0 \le k \le n-1,
$$
see Lemma 2.3 in \cite{B21}. Next, consider the sets $J_{n,s}$, $J_{n,b}$ defined in \eqref{eq71}, \eqref{eq72} with $\lambda(z)$ in place of $z$.
We have
\begin{align*}
\sum_{k \in \Lambda}|\phi_k^*(z) D_{\mu}(z)|^2
&\lesssim R_1(n)\sum_{k \in \Lambda}|\phi_k^*(\lambda(z)) D_{\mu}(\lambda(z))|^2\\
&\lesssim R_1(n)\Bigl(\sum_{k \in \Lambda\cap J_{n,s}} 1 + \sum_{k \in \Lambda\cap J_{n,b}}|\phi_k^*(\lambda(z)) D_{\mu}(\lambda(z))|^2\Bigr)\\
&\lesssim R_1(n)|\Lambda| + R_1(n)\sum_{k \in J_{n,b}}|\phi_k^*(\lambda(z)) D_{\mu}(\lambda(z))|^2.
\end{align*}
Now the claim follows from \eqref{eq46}, \eqref{eq47} applied to $\lambda(z)$ in place of $z$. \qed

\medskip

In the next proof we switch from the polynomials $\phi^{*}_{n}$ to polynomials $\tilde\phi^{*}_{\lambda, n}$ defined in \eqref{eq44star}. The fact that  polynomials  $\tilde\phi^{*}_{\lambda, n}$ have a small distortion near $\zeta$ (see \eqref{eq68}) will make it possible to get the asymptotics of $\phi^{*}_{n}(z)$ for 
$z \in \Gamma'_{\zeta, n}$ by using the asymptotics of $\phi^{*}_{n}(\lambda_n)$ at the point $\lambda_n = (1 - 1/n)\zeta$ in the region $\Gamma_{\zeta, n}$. The latter asymptotics was already found in \eqref{eq101}.

\medskip

Given $\lambda \in \D$ and $\eps \in (0, 1)$, let us consider the sets $I_{n, s}(\eps)$, $I_{n, b}(\eps)$ defined by
\begin{align}
I_{n, s}(\eps) 
&= \left\{0 \le k \le n-1:\;\;  |f_{k-1}(\lambda)| \le \eps\right\}, \label{eq64}\\ 
I_{n, b}(\eps) 
&= \left\{0 \le k \le n-1:\;\;  |f_{k-1}(\lambda)| > \eps\right\}, \label{eq65}
\end{align}
where, as before $f_{-1} = 0$. By construction, see \eqref{eq74}, for every $k \in I_{n, s}(\eps)$ we have
\begin{equation}\label{eq67}
\Biggl|\left|\frac{\tilde\phi^{*}_{\lambda, k}(z)}{\phi^{*}_{k}(z)}\right|^2 - 1\Biggr| \lesssim \eps, \qquad z \in \D.
\end{equation}

\medskip

\begin{Lem}\label{l39}
Let $\mu\in\szc$, $\zeta \in \T$, be such that \eqref{eq21} holds and $|D_\mu|$ has a non-zero finite non-tangential limit at $\zeta$, let $n_0$ be defined by \eqref{nova}, $\lambda_n = (1-1/n)\zeta$. For $\varepsilon >0$, consider the set $I_{n,s}(\eps)$ from \eqref{eq64} generated by $\lambda = \lambda_n$. We have
$$
\limsup_{n\to \infty}\sup_{z \in \Gamma'_{\zeta,n}}\left(\frac{1}{n}\sum_{k \in I_{n,s}(\eps)}\Bigl||\phi_k^*(z) D_{\mu}(z)|^2 - 1\Bigr|\right)  \lesssim \eps.
$$
\end{Lem}
\beginpf  At first, observe that for every $k \in I_{n,s}(\eps)$ and $z \in \D$ relation \eqref{eq67} implies 
$$
\Bigl||\phi_k^*(z) D_{\mu}(z)|^2 - |\tilde\phi_{\lambda_n, k}^*(z) D_{\mu}(z)|^2 \Bigr| \lesssim \eps  |\phi_k^*(z) D_{\mu}(z)|^2.
$$
Therefore, we have
$$
\sum\limits_{k \in I_{n,s}(\eps)}\Bigl||\phi_k^*(z) D_{\mu}(z)|^2 - 1\Bigr| \lesssim 
\sum\limits_{k \in I_{n,s}(\eps)}\Bigl||\tilde\phi_{\lambda_n, k}^*(z) D_{\mu}(z)|^2 - 1\Bigr| + \eps\sum\limits_{k = 0}^{n-1}|\phi_k^*(z) D_{\mu}(z)|^2.
$$
Recall that 
$$
\frac{1}{n}\sum\limits_{k = 0}^{n-1}|\phi_k^*(z) D_{\mu}(z)|^2 \le 1 + \delta_n(\mu,z)
$$
by \eqref{eq45}, and $\lim_{n \to \infty}\sup_{z \in \Gamma'_{\zeta,n}}\delta_n(\mu,z) = 0$. Thus, we only need to check that 
$$
\limsup_{n\to \infty}\sup_{z \in \Gamma'_{\zeta,n}}\left(\frac{1}{n}\sum\limits_{k \in I_{n,s}(\eps)}\Bigl||\tilde\phi_{\lambda_n, k}^*(z) D_{\mu}(z)|^2 - 1\Bigr|\right) = 0.
$$
For $z \in \Gamma'_{\zeta,n}$ and $0 \le k \le n$, $n \ge n_0$, with $n_0$ from \eqref{nova}, relation \eqref{eq68} gives
$$
\frac{1+O(\sqrt{\K(\mu, \lambda_n)})}{R_1(n)} \le \left|\frac{\tilde\phi_{\lambda_n, k}^*(z) D_{\mu}(z)}{\tilde\phi_{\lambda_n, k}^*(\lambda_n) D_{\mu}(\lambda_n)}\right|^2 
\le R_1(n) \cdot (1+O(\sqrt{\K(\mu, \lambda_n)})),
$$
where $R_1(n)$ is defined in \eqref{eq69}, $\lim_{n \to \infty}R_1(n) = 1$. It follows that
$$
\sup_{z \in \Gamma'_{\zeta,n}}\Bigl||\tilde\phi_{\lambda_n, k}^*(z) D_{\mu}(z)|^2 - |\tilde\phi_{\lambda_n, k}^*(\lambda_n) D_{\mu}(\lambda_n)|^2\Bigr| \le \eps_{n}|\tilde\phi_{\lambda_n, k}^*(\lambda_n) D_{\mu}(\lambda_n)|^2,
$$
for some sequence $\eps_{n}$ such that $\lim_{n \to \infty}\eps_{n} = 0$.
By triangle inequality, we have
$$
\sup_{z \in \Gamma'_{\zeta,n}}\bigl||\tilde\phi_{\lambda_n, k}^*(z) D_{\mu}(z)|^2 - 1\bigr|
\lesssim \bigl||\tilde\phi_{\lambda_n, k}^*(\lambda_n) D_{\mu}(\lambda_n)|^2 - 1\bigr| + \eps_{n}|\tilde\phi_{\lambda_n, k}^*(\lambda_n) D_{\mu}(\lambda_n)|^2.
$$
Finally, Lemma \ref{lemSt-n} for the point $z = \lambda_n$ gives
$$
\frac{1}{n}\sum\limits_{k \in I_{n,s} (\varepsilon)}\bigl||\tilde\phi_{\lambda_n, k}^*(\lambda_n) D_{\mu}(\lambda_n)|^2 - 1\bigr|
\lesssim \delta_n(\mu,\lambda_n) + \eta_n(\lambda_n) + \etat_n(\lambda_n), 
$$ 
and the claim follows from Lemma \ref{l31}, Lemma \ref{l32}, and the fact that $\lambda_n \in \Gamma_{\zeta, n}$. \qed

\bigskip

\noindent{\bf Proof of Theorem \ref{thm1}.} The proof is a combination of previous results. Lemma~\ref{lemSt-n} gives
$$
\frac{1}{n}\sum_{k=0}^{n-1}\Bigl||\phi_k^*(z) D_{\mu}(z)|^2 - 1\Bigr| \lesssim \delta_n(\mu,z) + \eta_n(z) + \etat_n(z) , \quad z \in \D. 
$$
Therefore, we have
$$
\sup_{z \in \Gamma_{\zeta, n}}\frac{1}{n}\sum_{k=0}^{n-1}\Bigl||\phi_k^*(z) D_{\mu}(z)|^2 - 1\Bigr| \to 0, \qquad n\to \infty,
$$
by Lemma \ref{l31} and Lemma \ref{l32}. To deal with the region $\Gamma'_{\zeta, n} = \Gamma_{\zeta}\setminus\Gamma_{\zeta, n}$, we take $\eps >0$, the point $\lambda_n = (1- 1/n) \zeta$  and consider the sets $I_{n,s}(\eps)$, $I_{n,b}(\eps)$ from \eqref{eq64}, \eqref{eq65} corresponding to the points $\lambda_n$, that is, $I_{n,s} (\varepsilon) = \{0 \le k \le n-1 \colon |f_{k-1} (\lambda_n)| < \varepsilon \}$ and $I_{n,b} (\varepsilon) = \{0 \le k \le n-1 \colon |f_{k-1} (\lambda_n)| \geq \varepsilon \}$. Lemma \ref{l39} gives 
$$
\limsup_{n\to \infty}\sup_{z \in \Gamma'_{\zeta,n}}\left(\frac{1}{n}\sum_{k \in I_{n,s}(\eps)}\Bigl||\phi_k^*(z) D_{\mu}(z)|^2 - 1\Bigr|\right)  \lesssim \eps.
$$
Choosing $\Lambda = I_{n,b}(\eps)$ in Lemma \ref{l36}, we obtain 
$$
\sup_{z \in \Gamma'_{\zeta,n}}\left(\frac{1}{n}\sum_{k \in I_{n,b}(\eps)}\Bigl||\phi_k^*(z) D_{\mu}(z)|^2 - 1\Bigr|\right)  \lesssim \frac{R_1(n)|I_{n,b}(\eps)|}{n} + R_1(n) R_2(n).
$$
Recall that $R_1(n) \to 1$, $R_2(n) \to 0$ for every $\zeta\in \T$ such that \eqref{eq21} holds and $|D_\mu|$ has a non-zero finite non-tangential limit at $\zeta$. The definition \eqref{etatn} of $\etat_n$ implies
$$
\frac{|I_{n,b}(\eps)|}{n}\log\frac{1}{1-\eps} \le \etat_n(\lambda_n).
$$
Since $\lambda_n \in \Gamma_{\zeta,n}$, we have $\limsup_{n\to \infty} \frac{|I_{n,b}(\eps)|}{n} = 0$ by Lemma \ref{l32}. This gives us 
$$
\limsup_{n\to \infty}\sup_{z \in \Gamma'_{\zeta,n}}\Biggl(\frac{1}{n}\sum_{k \in I_{n,b}(\eps)}\Bigl||\phi_k^*(z) D_{\mu}(z)|^2 - 1\Bigr|\Biggr) = 0.
$$
Collecting the estimates, we obtain
$$
\limsup_{n\to \infty}\sup_{z \in \Gamma'_{\zeta,n}}\left(\frac{1}{n}\sum_{k=0}^{n-1}\Bigl||\phi_k^*(z) D_{\mu}(z)|^2 - 1\Bigr|\right) \lesssim \eps.
$$
The theorem follows by letting $\eps \to 0$. \qed

\medskip

\begin{figure}
\centering
\includegraphics[width=\linewidth]{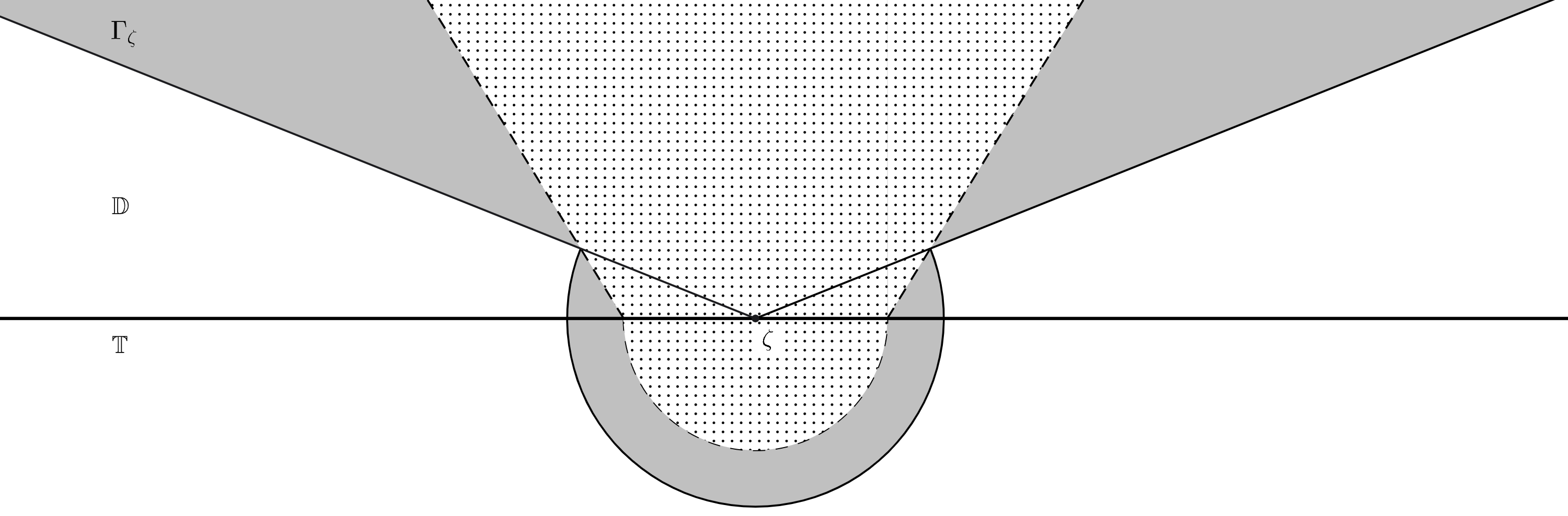}
\caption{The region appearing in Theorem \ref{thm1-bis} (filled with gray) contain the region from Figure \ref{fig:sub2} (filled here with dots).}
\label{fig:img4}
\end{figure}

We have proved Theorem \ref{thm1} for Stolz angles. The same method gives a bit more general result that we formulate below.

\begin{Thm}\label{thm1-bis}
Let $\mu \in \szc$ and $\zeta \in \T$ be such that \eqref{eq21} holds and $|D_\mu|$ has a non-zero finite non-tangential limit at $\zeta$. Fix $A > 0$ and let $B(\zeta, A/n) = \{z \in \C: \; |z-\zeta| \le A/n\}$. Then
\begin{equation}\label{thm-eq1bis}
\sup_{z \in \Gamma_\zeta \cup B(\zeta, A/n)}\Biggl(\frac{1}{n}\sum_{k = 0}^{n-1}\Bigl||\phi_k^*(z) D_{\mu}(\tilde z)|^2 - 1\Bigr|\Biggr) \to 0, \qquad n \to \infty,
\end{equation}
where we set $\tilde z = \zeta$ for $z \in B(\zeta, A/n)$ and $\tilde z = z$ for $z \in \Gamma_\zeta \setminus B(\zeta, A/n)$. 
\end{Thm}

\medskip

The regions $\Gamma_\zeta \cup B(\zeta, A/n)$ in Theorem \ref{thm1-bis} are shown on Figure \ref{fig:img4}. Note that the width of the Stolz angle $\Gamma_\zeta$ and the parameter $A>0$ are arbitrary, hence these regions contain the area shown on Figure \ref{fig:sub2}. This area is filled with dots on Figure~\ref{fig:img4}.

\medskip

\section{Proof of Theorem \ref{thm2}}
The proof of Theorem \ref{thm2} uses the same tools as the proof of Theorem \ref{thm1}: Khrushchev-type formula \eqref{eqKH} and the entropy function $\K(\mu, z)$. The latter controls Schur functions of $\mu$ that appear in \eqref{eqKH}, see Lemma \ref{l31} and Lemma \ref{l32}. Recall that we already know the averaged asymptotic behaviour of $|\phi_n^*D_{\mu}|$ under assumption \eqref{eq21}, that is, when $\lim_{r \to 1}\K(\mu, r\zeta) = 0$, see Theorem \ref{thm1}. The stronger assumption \eqref{eq22}, that we reproduce here,
\begin{equation}\tag{2.2}\label{eq22bis}
\sum_{n \ge 0} \K(\mu, z_n(\zeta)) < \infty, \qquad z_n(\zeta) = (1- 2^{-n}) \zeta,
\end{equation}
will allow us to pass from $|\phi_n^*D_{\mu}|$ to $\phi_n^*D_{\mu}$. For this we utilize the fact that $\phi_n^*D_{\mu}$ is an outer function in $\D$ with a positive value at $0$, hence $\log \phi_n^*D_{\mu}$ is a properly defined analytic function in $\D$ which is completely determined by its real part $\log |\phi_n^*D_{\mu}|$ via the operator of harmonic conjugation,
$$
\mathcal{Q}: u \mapsto \int_{\T}u(\xi) \Im\left(\frac{1 + \bar \xi z}{1 - \bar \xi z}\right) dm(\xi), \quad z \in \D . 
$$
Assumption \eqref{eq22} will allow us to prove that $\Im \log \phi_n^*D_{\mu} = \mathcal{Q}(\log|\phi_n^*D_{\mu}|)$ is small in $\Gamma_\zeta$, hence $\log \phi_n^*D_{\mu}$ is close to $\log |\phi_n^*D_{\mu}|$, which, in turn, is close to $0$. This will lead to the fact that $\phi_n^*D_{\mu}$ converges to $1$ in the strong Ces\`aro sense uniformly in $\Gamma_\zeta$, i.e.,  to the conclusion of Theorem \ref{thm2}. 

\medskip

\noindent Let us recall the definition of truncated cones 
$
\Gamma_{\zeta, n} = \{ z \in \Gamma_\zeta : |z - \zeta| \ge 1/n\}.
$
We start with a simple lemma.
\begin{Lem}\label{l00}
Let $u \in L^1(\T)$. Then 
\begin{equation}\notag
\lim_{n \to \infty} \frac{\sup_{z \in \Gamma_{\zeta, n}} |(\mathcal{Q}u)(z)|}{n}  = 0, \quad \zeta \in \T.
\end{equation}
\end{Lem}
\beginpf For every $u \in L^1(\T)$ and $n \ge 1$, we have 
\begin{equation}\label{eq121-2}
\frac{\sup_{z \in \Gamma_{\zeta, n}}|(\mathcal{Q}u)(z)|}{n} \le \sup_{z \in \Gamma_{\zeta, n}}\int_{\T} \frac{2|u(\xi)|}{n|1 - \bar\xi z|}\, dm(\xi) \le 2\|u\|_{L^1(\T)}.
\end{equation}
Let $u_\eps \in L^1(\T)$ be a function with $\mathcal{Q}u_\eps \in L^\infty(\D)$ and such that $\|u - u_\eps\|_{L^1(\T)} \le \eps$. Applying \eqref{eq121-2} to $u-u_\eps$, we obtain
\begin{align*}
\limsup_{n \to \infty} \frac{\sup_{z \in \Gamma_{\zeta, n}} |(\mathcal{Q}u)(z)|}{n} 
\le &\limsup_{n \to \infty} \frac{\sup_{z \in \Gamma_{\zeta, n}} |(\mathcal{Q}u- \mathcal{Q}u_\eps)(z)|}{n}+ \\
&+ \limsup_{n \to \infty} \frac{\sup_{z \in \Gamma_{\zeta, n}} |(\mathcal{Q}u_\eps)(z)|}{n} \le 2\eps .
\end{align*}
Since $\eps> 0$ is arbitrary, the lemma follows. \qed

\medskip

\begin{Lem}\label{lemImBis}
Suppose that $\mu \in \szc$ and $\zeta \in \T$ are such that \eqref{eq22bis} holds. Then 
\begin{equation}\notag
\lim_{n \to \infty} \sup_{z \in \Gamma_{\zeta, n}} \frac{1}{n}\sum_{k=0}^{n-1}|\Im \log O_{k}(z)|  = 0.
\end{equation}
\end{Lem}
\beginpf We will use outer functions $O_n$ from Lemma \ref{l1}. Set $\alpha_k(z) = \Im\log O_k(z)$ for $z\in \D$, $k \ge 0$. Given a positive integer $n$, let $K_n$ be the integer part of $\log_2 n$. We first show that there exists a constant $c(\mu,\zeta)>0$ such that  
\begin{equation}
\label{eq116}
\frac{1}{n}\sum_{k=0}^{n-1}|\alpha_k(z_{j+1}) - \alpha_k(z_j)| \le c(\mu,\zeta)\K(\mu, z_j),  \qquad
1\le j \le K_n,
\end{equation}
where we set $z_j = z_j(\zeta) = (1 - 2^{-j})\zeta$ for brevity. Fix $1 \le j \le K_n$. We have  
\begin{align*}
|\alpha_k(z_{j+1}) - \alpha_k(z_j)| 
&\le \frac{1}{2}\int_{\T}|\log(1-|f_k|^2)| \left|\frac{1+\bar\xi  z_{j+1}}{1-\bar\xi z_{j+1}} - \frac{1+\bar\xi  z_{j}}{1-\bar\xi z_j} \right|dm\\
&\lesssim \int_{\T}|\log(1-|f_k|^2)|\frac{1 - |z_{j}|^2}{|1-\bar\xi z_j|^2}dm\\
&=  \K(\mu_k , z_j) -  \log (1- |z_j f_k (z_j)|^2). 
\end{align*}
Last identity is formula \eqref{entropyschur-n}. Since by \eqref{decayentropy} we have $\K (\mu_k , z ) \le \K (\mu , z)$ for any $z \in \D$, summing up over $k$ for fixed $j$ and using $|z_j f_k (z_j)| \le |f_k (z_j)|$,  we obtain
$$
\frac{1}{n}\sum_{k=0}^{n-1}|\alpha_k(z_{j+1}) - \alpha_k(z_j)| \lesssim \K(\mu, z_j) -\frac{1}{n} \sum_{k=0}^{n-1} \log (1- |f_k (z_j)|^2).
$$
As in \eqref{eqIns} and \eqref{eqInb}, consider the sets $I_{n, s} 
= \left\{0 \le k \le n-1 \colon  |f_k(z_j)| \le 1/2\right\}$ and $I_{n, b} = \left\{0 \le k \le n-1 \colon  |f_k(z_j)| > 1/2\right\}$. Since $j \le K_n$, we have $1/ n \lesssim 1- |z_j|^2$ and then 
\begin{align*}
-\frac{1}{n} \sum_{k \in I_{n,s}} \log (1- |f_k (z_j)|^2) 
& \asymp \frac{1}{n} \sum_{k \in I_{n,s}} |f_k (z_j)|^2 \\
& \lesssim \sum_{k=0}^{n-1} (1-|z_j|^2) |f_k (z_j)|^2 \lesssim \K (\mu , z_j ),
\end{align*}
where in the last estimate we have used \eqref{entf}. Next we consider the contribution coming from indices in $I_{n, b}$. Since $c_1(\mu,\zeta) = \sup_{j \ge 0}\K (\mu , z_j )$ is finite, formula \eqref{entf} gives 
$$
\sup_{k,j} \frac{(1- |z_j|^2) |f_k (z_j)|^2}{1 - |f_k (z_j)|^2} \le e^{c_1(\mu,\zeta)}. 
$$
It follows that for $k \in I_{n, b}$ we have
\begin{align*}
- (1- |z_j|^2) \log (1 - |f_k (z_j)|^2) 
&\le \frac{1- |z_j|^2}{1 - |f_k (z_j)|^2} \le \frac{1}{4}\frac{(1- |z_j|^2)|f_k (z_j)|^2}{1 - |f_k (z_j)|^2} \le \\
&\le c_2(\mu,\zeta) \log \left(1 + \frac{(1- |z_j|^2) |f_k (z_j)|^2}{1 - |f_k (z_j)|^2} \right),  
\end{align*}
with the constant 
$$
c_2(\mu,\zeta) = \sup_{0\le y \le e^{c_1(\mu,\zeta)}} \frac{y/4}{\log(1 + y)}.
$$ 
Since $1/n \lesssim 1 - |z_j|^2$, we deduce from \eqref{entf} the bound
\begin{equation*}
-\frac{1}{n+1} \sum_{k \in I_{n,b}} \log (1- |f_k (z_j)|^2) \le c_2(\mu,\zeta)\K(\mu , z_j). 
\end{equation*}
This finishes the proof of \eqref{eq116}. Now using \eqref{eq116} and the fact that $\alpha_k(z_0) = \Im \log O_k (0) = 0$ for all $k\ge 0$, a telescopic sum argument shows that
\begin{equation}\label{112}
\frac{1}{n}\sum_{k=0}^{n-1}|\alpha_k(z_{j+1})| \le c(\mu, \zeta) \sum_{s = 0}^{j}\K(\mu, z_s), \qquad 1 \le j \le K_n. 
\end{equation}
Next, given $ z \in \Gamma_{\zeta , n}$, let $j(z) \ge 1$ be the integer such that $|z_{j(z)}| < |z| \le |z_{j(z)+1}|$. Repeating the proof of \eqref{eq116}, we obtain
\begin{equation*}
\frac{1}{n} \sum_{k=0}^{n-1} |\alpha_k(z) - \alpha_k(z_{j(z)})| \le c(\mu, \zeta)\K(\mu, z_{j(z)}), 
\end{equation*}
with a possibly larger constant $c(\mu, \zeta)$ that still does not depend on $n$ and $z \in \D$. Then \eqref{112} gives 
$$
\frac{1}{n}\sum_{k=0}^{n-1}|\alpha_k(z)| \le c(\mu, \zeta)\sum_{s = 0}^{j(z)}\K(\mu, z_s).
$$
Recall that measures $\mu_n$, $n \ge 0$, were defined using the Schur's algorithm \eqref{eq41}. Since $\K(\mu_n, z) \le \K(\mu, z)$ for every $n \ge 0$, $z \in \D$ by \eqref{decayentropy}, our construction gives $c(\mu_n, \zeta) \le c(\mu, \zeta)$ for every $n \ge 0$. Now take a large number $\ell \ge 0$ and apply the first part of the proof to $\mu_\ell$ in place of $\mu$. We obtain
\begin{equation}\label{eq12}
\frac{1}{n}\sum_{k=0}^{n-1}|\alpha_{\ell+k}(z)| \le c(\mu_\ell, \zeta) \sum_{s = 0}^{j(z)}\K(\mu_\ell, z_s) \le c(\mu, \zeta)\sum_{s = 0}^{\infty}\K(\mu_\ell, z_s).
\end{equation}
Note that  the last expression tends to $0$ as $\ell \to \infty$. Indeed, assumption \eqref{eq22bis} implies that $\sum_{s = s_0}^{\infty}\K(\mu_\ell, z_s) \le \sum_{s = s_0}^{\infty}\K(\mu, z_s) \to 0$ as $s_0 \to \infty$, and we have shown in the proof Lemma \ref{l31} that $\K(\mu_\ell, z_s) \to 0$ for each fixed point $z_s$ as $\ell \to \infty$ (note that  Lemma \ref{l31} holds under the assumption \eqref{eq21} which is weaker than \eqref{eq22}). Since for each fixed $\ell \ge 0$ we have
$$
\lim_{n \to \infty}\sup_{z \in \Gamma_{\zeta,n}}\frac{1}{n}\sum_{s=0}^{\ell}|\alpha_{s}(z)| \le \sum_{s=0}^{\ell}\lim_{n \to \infty}\frac{\sup_{z \in \Gamma_{\zeta,n}} |\alpha_{s}(z)|}{n} = 0,
$$ 
by Lemma \ref{l00}, relation \eqref{eq12} implies the claim. \qed

\medskip

We are now ready to prove the uniform convergence on the truncated cones $\Gamma_{\zeta, n}$. 

\begin{Lem}\label{uniftruncated1}
Suppose that $\mu \in \szc$ and $\zeta \in \T$ are such that \eqref{eq22bis} holds. Then 
\begin{equation}\notag
\lim_{n \to \infty} \sup_{z \in \Gamma_{\zeta, n}} \frac{1}{n}\sum_{k=0}^{n-1}|\varphi_k^* (z) D_\mu (z)-1| =0.
\end{equation}
\end{Lem}
\beginpf For $z \in \Gamma_{\zeta , n}$, consider the sets of indices 
\begin{align*}
J_{n, s} &= \{0 \le k \le n-1 \colon |\varphi_k^* (z) D_\mu (z)| \le 2\},\\
J_{n, b} &= \{0 \le k \le n-1 \colon |\varphi_k^* (z) D_\mu (z)| > 2\}. 
\end{align*}
By Lemma \ref{lemSt-n}, we have
\begin{align*}
\frac{1}{n}\sum_{k \in J_{n, b}}\Bigl|\phi_k^*(z) D_{\mu}(z) - 1\Bigr| 
&\lesssim \frac{1}{n}\sum_{k=0}^{n-1}\Bigl||\phi_k^*(z) D_{\mu}(z)|^2 - 1\Bigr|\\
&\lesssim \delta_n(\mu,z) + \eta_n(z) + \etat_n(z), 
\end{align*}
which by Lemma \ref{l31} and Lemma \ref{l32} tends to $0$ uniformly on $\Gamma_{\zeta , n}$ as $n \to \infty$. Next, Lemma \ref{lemAv} says that 
\begin{align*}
\frac{1}{n}\sum_{k=0}^{n-1} \Bigl| \log(1 - z b_k(z) f_k(z)) \Bigr| \le \etat_n(z), \quad z \in \D , & \\
\frac{1}{n}\sum_{k=0}^{n-1}\log|O_k(z)|^{-2} \le \eta_n(z) + \etat_n(z), \quad  z \in \D. 
\end{align*}
We also have $\varphi_k^* D_\mu = O_k (1-z b_k f_k)^{-1}$ by Lemma \ref{l1}. Using the elementary estimate 
$$
|r e^{i \theta} - 1| \le |r - 1| + |e^{i\theta} - 1| \lesssim |\log r| + |\theta|,
$$ 
for $0<r\le 2$, $|\theta| \le \pi$, we  get for $k \in J_{n, s}$
\begin{align*}
|\varphi_k^* D_\mu -1 | 
&\lesssim \bigl|\log |O_k(1-z b_k f_k)^{-1}|\bigr| + \bigl|\Im\log O_k (1-z b_k f_k)^{-1}\bigr|, \\
&\le \bigl|\log |O_k|\bigr| + 2\bigl|\log (1-z b_k f_k)^{-1}\bigr|+|\Im\log O_k|, \\
&\le \log |O_k|^{-2} + 2\bigl|\log (1-z b_k f_k)\bigr|+|\Im\log O_k|. 
\end{align*}
Then 
\begin{equation*}
    \frac{1}{n} \sum_{k \in J_{n, s}} |\varphi_k^* (z) D_\mu (z) -1 | \lesssim \eta_n (z) + \etat_n (z) + \frac{1}{n} \sum_{k \in J_{n, s}} |\Im \log O_k (z)|,
\end{equation*}
which by Lemma \ref{l31}, Lemma \ref{l32} and Lemma \ref{lemImBis} tends to $0$ uniformly on $\Gamma_{\zeta, n}$. This finishes the proof. \qed

\medskip

Next we focus on the uniform convergence on $\Gamma'_{\zeta, n} = \Gamma_\zeta \setminus \Gamma_{\zeta, n}$. 

\begin{Lem}\label{uniftruncated2}
Let $\mu \in \szc$ and $\zeta \in \T$ be such that $D_\mu$ has a non-zero finite non-tangential limit at $\zeta$. Assume that condition \eqref{eq22} holds at $\zeta$. Then 
\begin{equation}\notag
\lim_{n \to \infty} \sup_{z \in \Gamma'_{\zeta , n}} \frac{1}{n}\sum_{k=0}^{n-1}|\phi_k^*(z) D_\mu (z)-1| =0.
\end{equation}
\end{Lem}
\beginpf For $z \in \Gamma'_{\zeta , n}$, let $J_{n, b}(z) = \{0 \le k \le n-1 \colon |\varphi_k^* (z) D_\mu (z)| > 2\}$. Observe that
$$
\frac{1}{n}\sum_{k \in J_{n, b}(z)} |\varphi_k^* (z) D_\mu (z)-1| \lesssim \frac{1}{n}\sum_{k \in J_{n, b}(z)}\bigl| |\varphi_k^*(z) D_\mu (z)|^2 - 1 \bigr|, 
$$
which by Theorem \ref{thm1} tends to $0$ uniformly on $\Gamma_\zeta$. Set $\lambda_n = (1- n^{-1}) \zeta$ for $n \ge 1$. Fix $\eps>0$. Recall that we put $f_{-1}(z) = 0$ for all $z \in \D$. Consider the sets of integers 
\begin{align*}
I_{n,s}(\eps) &= \{0 \le k \le n-1 \colon |f_{k-1} (\lambda_n)|\le \eps\}, \\
I_{n,b}(\eps) &= \{0 \le k \le n-1 \colon |f_{k-1} (\lambda_n)|> \eps\}.
\end{align*}
We have 
\begin{equation}\notag
\frac{1}{n} \sum_{k \in I_{n , b}(\eps) \setminus J_{n, b}(z)} |\varphi_k^* (z) D_\mu (z)- 1| \le 
\frac{1}{n} \sum_{k \in I_{n , b}(\eps) \setminus J_{n, b}(z)} 3 \le 3\frac{|I_{n , b}|}{n}. 
\end{equation}
The definition \eqref{etatn} of $\etat_n$ gives  
\begin{equation}\label{primer}
\frac{|I_{n , b}(\eps)|}{n} \lesssim \frac{\etat_n(\lambda_n)}{|\log (1 - \varepsilon)|}, 
\end{equation}
which by Lemma \ref{l32}, tends to $0$ as $n \to \infty$. We see that
\begin{equation}\label{eq122}
\lim_{n \to \infty}\sup_{z \in \Gamma'_{\zeta, n}}\frac{1}{n}\sum_{k \in I_{n,b} (\varepsilon) \setminus J_{n,b}(z)} \left| \varphi_k^* (z) D_\mu(z) -1 \right| = 0
\end{equation}
for every $\eps > 0$. On the other hand, formula \eqref{varphi} gives 
\begin{equation}
    \label{nou}
    \left|\frac{\tilde\phi_{\lambda_n , k}^*(z)}{\phi_{k}^*(z)} - 1 \right| \lesssim \varepsilon, \qquad k \in I_{n, s} (\varepsilon). 
\end{equation}
By Lemma \ref{domain-bound}, see \eqref{eq115bis}, there exists $n_0(\eps)\ge 1$ such that
\begin{equation}\label{segon}
\sup_{z \in \Gamma'_{\zeta , n}}\left|\frac{\tilde\phi_{\lambda_n, k}^*(z)}{\tilde\phi_{\lambda_n, k}^*(\lambda_n)} - 1 \right| \lesssim \varepsilon, \quad 0 \le k \le n, 
\end{equation}
for every $n \ge n_0(\eps)$. Since $D_\mu$ has a non-zero finite non-tangential limit at $\zeta$, we also may assume that
\begin{equation}
    \label{tercer}
    \sup_{z \in \Gamma'_{\zeta , n}} \left|\frac{D_\mu (z)}{D_\mu (\lambda_n)} - 1 \right| \le \eps, \quad n \ge n_0(\eps).
\end{equation}
For $\eps > 0$ small enough, a combination of \eqref{nou}, \eqref{segon} and \eqref{tercer} gives 
$$
\sup_{n\ge n_0(\eps)}\sup_{z \in \Gamma'_{\zeta,n}}\sup_{k \in I_{n,s}(\eps)\setminus J_{n,b}(z)}\left| \varphi_k^* (z) D_\mu (z) - \varphi_k^* (\lambda_n) D_\mu (\lambda_n)\right| \lesssim \eps. 
$$
Then,
$$
\frac{1}{n} \sum_{k \in I_{n,s} (\varepsilon) \setminus J_{n,b}(z)} \left| \varphi_k^* (z) D_\mu (z) -1 \right| \lesssim \frac{1}{n} \sum_{k \in I_{n,s} (\varepsilon) \setminus J_{n,b}(z)} \left| \varphi_k^* (\lambda_n) D_\mu (\lambda_n) -1 \right| + \eps, 
$$
for $n \ge n_0(\eps)$, $z \in \Gamma'_{\zeta,n}$. Since $\lambda_n \in \Gamma_{\zeta,n}$, we have 
$$
\lim_{n \to \infty}\frac{1}{n}\sum_{k \in I_{n,s} (\varepsilon) \setminus J_{n,b}(z)} \left| \varphi_k^* (\lambda_n) D_\mu (\lambda_n) -1 \right| = 0,
$$
for every $\eps > 0$ by Lemma \ref{uniftruncated1}. Thus, for any small enough $\varepsilon >0$ we have
\begin{equation}\label{eq123}
\limsup_{n \to \infty}\sup_{z \in \Gamma'_{\zeta, n}}\frac{1}{n} \sum_{k \in I_{n,s} (\varepsilon) \setminus J_{n,b}(z)} \left| \varphi_k^* (z) D_\mu (z) -1 \right| \lesssim \eps .
\end{equation}
Summing up \eqref{eq122}, \eqref{eq123}, and sending $\eps \to 0$, we conclude the proof. \qed

\medskip

\noindent{\bf Proof of Theorem \ref{thm2}.} The result follows from Lemma \ref{uniftruncated1} and Lemma \ref{uniftruncated2}. \qed

\medskip

\section{Proof of Theorem \ref{thm3}}
\begin{Lem}\label{lemChain}
Let $\mu \in \szc$, $\zeta\in \T$ be such that \eqref{eq23} holds. Consider the Schur functions $f_n$ of $\mu$, see \eqref{eq41}. 
We have $\lim\limits_{n \to \infty} \sup\limits_{ z \in \Gamma_\zeta } |f_{n}(z) | = 0$.
\end{Lem}
\beginpf Recall the notation $z_k = z_k(\zeta) = (1 - 2^{-k}) \zeta$, $k\ge 0$. We have
\begin{align*}
|f_n(z_{k+1}) - f_n(z_k)| 
&\le\Pp(|f_n - f_n(z_k)|, z_{k+1}) \lesssim \Pp(|f_n - f_n(z_k)|, z_{k}) \lesssim \sqrt{\K(\mu_n, z_k)},
\end{align*}
see Theorem 2 of \cite{BD21} for the last estimate. Thus, for every $j \ge 1$ we have
$$
|f_n(z_j) - f_n(z_0)| \le \sum_{k=0}^{j-1}|f_n(z_{k+1}) - f_n(z_k)| \lesssim \sum_{k=0}^{j-1}\sqrt{\K(\mu_n, z_k)}.
$$
Given $z \in \Gamma_\zeta$ pick the positive integer $j(z)$ such that $2^{-j(z)-1} \le 1- |z| < 2^{-j(z)}$. A variant of the previous argument shows $|f_n (z) - f_n (z_{j(z)})| \lesssim \sqrt{\K(\mu_n , z_{j(z)})}$. Hence
\begin{equation}\label{sumsqrt}
|f_n(z)|  \lesssim \sum_{k=0}^{\infty}\sqrt{\K(\mu_n, z_k)} + |f_n (z_0)|. 
\end{equation}
By Szeg\H{o} theorem, see \eqref{szego}, we have
\begin{equation}\label{eq5}
\lim_{n \to \infty} \int_{\T} \log(1-|f_n|^2) dm = 0.
\end{equation}
It follows that $f_n \to 0$ in Lebesgue measure on $\T$  as $n \to \infty$. In particular, we have $\lim_{n \to \infty}f_n(z) = 0$ for every $z \in \D$. Moreover, \eqref{eq5} yields $\lim_{n \to \infty}\K(\mu_n, z) = 0$ for each $z \in \D$, see \eqref{entropyschur-n} Since $\K(\mu_{n}, z) \le \K(\mu, z)$ at every $z \in \D$ by \eqref{decayentropy}, we now derive from \eqref{eq23} that
\begin{align}
\limsup_{n \to \infty }\sum_{k = 0}^\infty \sqrt{\K(\mu_{n}, z_k)} 
&= \lim_{\ell \to \infty}\limsup_{n \to \infty}\sum_{k = \ell}^\infty \sqrt{\K(\mu_n, z_k)} \notag\\
&\le \lim_{\ell \to \infty }\sum_{k = \ell}^\infty \sqrt{\K(\mu, z_k)} = 0, \label{eq4}
\end{align}
which completes the proof. \qed

\medskip

\begin{Lem}\label{lemRe}
Let $\mu \in \szc$, $\zeta\in \T$ be such that \eqref{eq23} holds. Consider the outer functions $O_n$, $n \ge 0$, from Lemma \ref{l1}. We have
$$
\lim_{n \to \infty} \sup_{z \in \Gamma_\zeta} |\Re \log O_{n}(z)| = 0.
$$
\end{Lem}
\beginpf Formula \eqref{entropyschur} says that
\begin{equation}\label{logmod}
- \log |O_n (z)|^2 = \Pp (\log (1- |f_n|^2) , z) = \K (\mu_n , z) - \log (1- |zf_n (z)|^2), 
\end{equation}
for any $ z \in \D$ and any $n\ge 0$. As we have seen in the proof of Lemma \ref{lemChain}, $\K (\mu_n , z)$ tends to $0$ uniformly on compacts in $\D$. Since $\K (\mu_n , z) \le \K (\mu , z)$ which tends to $0$ as $z$ non-tangentially tends to $\zeta$, we obtain
\begin{equation}\label{eq127}
\lim_{n \to \infty } \sup_{z \in \Gamma_\zeta} \K(\mu_n , z) = 0. 
\end{equation}
By Lemma \ref{lemChain}, we have $\log (1- |zf_n (z)|^2) \to 0$ uniformly on $\Gamma_\zeta$. Now formula \eqref{logmod} and the fact that 
$\Re\log O_n = \log |O_n|$ finish the proof.
\qed

\begin{Lem}\label{lemIm}
Let $\mu \in \szc$, $\zeta\in \T$ be such that \eqref{eq23} holds. Consider the outer functions $O_n$, $n \ge 0$, from Lemma \ref{l1}. We have
\begin{equation}\label{eq7}
\lim_{n \to \infty} \sup_{z \in \Gamma_\zeta} |\Im \log O_{n}(z) | = 0.
\end{equation}
\end{Lem}
\beginpf We use again notation $z_k = (1- 2^{-k}) \zeta$, $k \ge 0$. Let us first check that there exists a constant $c(\mu, \zeta)$ depending only on $\mu$, 
$\zeta$ such that  
\begin{equation}\label{eq6}
|\Im\log O_n(z_{k+1}) - \Im \log O_n(z_k)| \le c(\mu, \zeta) \sqrt{\K(\mu_n, z_k)}, 
\end{equation}
for every $k\ge 0$, $n\ge 0$. Take a constant $\lambda \in \R$. Since the kernel $\Im (1+\bar\xi  z)(1-\bar\xi z)^{-1}$ is odd, we have 
$$
\Im \log O_n(z) =  \Im \frac{1}{2}\int_{\T} (\log(1-|f_n|^2) - \lambda)\frac{1+\bar\xi  z}{1-\bar\xi z}\,dm(\xi), \qquad z \in \D. 
$$
Then  
\begin{align*}
& |\Im \log O_n(z_{k+1}) - \Im \log O_n(z_k)| \le \\ 
&\le \frac{1}{2}\int_{\T}|\log(1-|f_n(\xi)|^2) - \lambda| \left|\frac{1+\bar\xi  z_{k+1}}{1-\bar\xi z_{k+1}} - \frac{1+\bar\xi  z_{k}}{1-\bar\xi z_k} \right|\,dm(\xi) \\
& \lesssim \int_{\T}|\log(1-|f_n(\xi)|^2) - \lambda |\frac{1 - |z_{k}|^2}{|1-\bar\xi z_k|^2} \,dm(\xi) . 
\end{align*}
Now, for the special choice $\lambda = \Pp (\log (1-|f_n|^2), z_k)$, we have
\begin{equation}\label{eq125}
\int_{\T}|\log(1-|f_n|^2) - \lambda |\frac{1 - |z_{k}|^2}{|1-\bar\xi z_k|^2} \,dm \lesssim \max(\K(\mu_n, z_k), \sqrt{\K(\mu_n, z_k)}), 
\end{equation}
by Lemma 3 in \cite{BD21} (see also page 12 in \cite{BD21}). Since 
$$
\sup_{n,\,k \ge 0}\K(\mu_n, z_k) \le \sup_{z \in \Gamma_\zeta}\K(\mu, z) < \infty , 
$$ 
there is a constant $c(\mu, \zeta)$ such that the right hand side in \eqref{eq125} does not exceed $c(\mu, \zeta)\sqrt{\K(\mu_n, z_k)}$ for every $k\ge 0$, $n\ge 0$. So, \eqref{eq6} is proved. Next, given $z \in \Gamma_\zeta$, pick the positive integer $j(z)$ such that $2^{-j(z)-1} \le 1-|z| < 2^{-j(z)}$. Similarly to  \eqref{eq6}, we have
\begin{equation}\label{eq126}
|\Im\log O_n(z) - \Im \log O_n(z_{j(z)})| \le c(\mu, \zeta) \sqrt{\K(\mu_n, z_{j(z)})}. 
\end{equation}
Using \eqref{eq6}, \eqref{eq126}, and $\Im \log O_n (z_0) =0$, a telescopic sum argument gives
\begin{equation}\label{eq128}
|\Im \log O_n(z)| \le c(\mu, \zeta) \sum_{k = 0}^{j(\zeta)} \sqrt{\K(\mu_{n}, z_k)} \le c(\mu, \zeta) \sum_{k = 0}^\infty \sqrt{\K(\mu_{n}, z_k)}.
\end{equation}
Recall that $\lim_{n \to \infty }\sum_{k = 0}^\infty \sqrt{\K(\mu_{n}, z_k)} = 0$ under assumption \eqref{eq23}, see \eqref{eq4}. We now see that
\eqref{eq7} follows from \eqref{eq128}. \qed

\bigskip

\noindent{\bf Proof of Theorem \ref{thm3}.} The claim of the theorem follows immediately from Lemma \ref{lemChain}, Lemma \ref{lemRe}, and Lemma \ref{lemIm} using formula $\phi_n^* D_{\mu} = O_{n} / (1 - zb_n f_n)$ from Lemma \ref{l1}. \qed

\medskip

\section{Proof of Theorem \ref{thm4}}
Recall that we study the asymptotic behaviour of $k_{\mu, n}(z_1,z_2)$, $z_{1,2} \in \Gamma_{\zeta}$, by estimating the function $r_{n}$ 
defined by 
\begin{equation}\label{universality0-thm4-bis}
k_{\mu, n}(z_1, z_2) = \ov{D_{\mu}^{-1}(z_2)}D_{\mu}^{-1}(z_1) \frac{1 - \bar z_2^{n} z_1^{n}}{1 - \bar z_2 z_1}\bigl(1 + r_n(z_1, z_2)\bigr).
\end{equation}  
Given $\delta > 0$, $\Delta > 0$, define 
\begin{align}
\Gamma_{\zeta,n}(\delta, \Delta) &= \{z \in \Gamma_\zeta:\; \Delta \le |z| \le 1 - \delta/n\}, \label{eq871}\\
\Gamma'_{\zeta,n}(\delta) &= \{z \in \Gamma_\zeta:\; |z| > 1 - \delta/n\}, \label{eq872}\\
\Gamma''_{\zeta}(\Delta) &= \{z \in \Gamma_\zeta:\; |z| < \Delta\}. \label{eq873}
\end{align}
We first consider the ``diagonal'' case of Theorem \ref{thm4}. 
\begin{Lem}\label{l51}
Let $\mu \in \szc$ and $\zeta \in \T$ be such that $|D_\mu|$ has a non-zero finite non-tangential limit at $\zeta$. If \eqref{eq21} holds at $\zeta$, we have 
\begin{gather}
k_{\mu, n}(z,z) = |D^{-1}_\mu(z)|^{2}\frac{1 - |z|^{2n}}{1 - |z|^2}\bigl(1 + R_n(z)\bigr), \label{eq81} \\
\sup_{z \in \Gamma_{\zeta}}\bigl|R_n(z)\bigl| \to 0, \quad n \to \infty. \label{eq130}
\end{gather}
\end{Lem}
\beginpf It follows from general Szeg\H{o} theory (see \eqref{szthm} and \eqref{reprokernel}) that for every $\mu$ in $\szc$ we have 
$$
k_{\mu, n}(z_1, z_2) \to \frac{D_{\mu}(z_1)^{-1}\ov{D_{\mu}(z_2)^{-1}}}{1 - \bar z_1 z_2}, \quad n \to \infty , 
$$ 
on compact subsets of $\D \times \D$. Hence, there exists an increasing sequence $\Delta_n \to 1$ such that 
$$
\sup_{z \in \Gamma''_{\zeta}(\Delta_n)}\bigl|R_n(z)\bigl| \to 0, \quad n \to \infty,
$$
for the functions $R_n$ defined by \eqref{eq81}. We can also choose a slowly increasing sequence $\delta_n \to \infty$, $\delta_n = o(n)$, such that 
$$
\sup_{z \in \Gamma'_{\zeta,n}(\delta_n)}\bigl|R_n(z)\bigl| \to 0, \quad n \to \infty.
$$
Indeed, the existence of such a sequence is a consequence of Theorem 1.1 in \cite{Bes24} and M\'at\'e, Nevai, and Totik asymptotic relation \eqref{eq32}. Note that \eqref{eq32} holds under the assumptions of Theorem \ref{thm4} by Theorem \ref{thm1}. Thus, to prove \eqref{eq130} we need to check that $\sup_{z \in \Gamma_{\zeta,n}(\delta_n, \Delta_n)}|R_n(z)| \to 0$ as $n \to \infty$. Moreover, since $\sup_{z \in \Gamma_{\zeta,n}(\delta_n, \Delta_n)} |z|^{2n} \to 0$ as $n \to \infty$, we actually need to prove that
\begin{equation}\label{eq777}
\sup_{z \in \Gamma_{\zeta,n}(\delta_n, \Delta_n)}\bigl|R^*_n(z)\bigl| \to 0, \quad n \to \infty,
\end{equation}
for functions $R^*_n$ defined by
$$
k_{\mu, n}(z,z) = \frac{|D^{-1}_\mu(z)|^{2}}{1 - |z|^2}\bigl(1 + R_n^*(z)\bigr).
$$
A computation of Fourier coefficients similar to \eqref{eq144} leads to the well-known formulas
\begin{align}
k_{\mu, n}(z_1,z_2) &= \sum_{k=0}^{n-1} \phi_k (z_1)\ov{\phi_k (z_2)}, \label{eq145}\\
\frac{D_\mu^{-1}(z_1)\ov{D_\mu^{-1}(z_2)}}{1 - \bar z_1 z_2} &= \sum_{k=0}^{\infty} \phi_k (z_1)\ov{\phi_k (z_2)}, \label{eq146}
\end{align}
for $z_1, z_2 \in \D$. In particular, 
\begin{equation}\label{eq129}
k_{\mu, n}(z,z) \le \frac{|D^{-1}_\mu(z)|^{2}}{1 - |z|^2}, \qquad z \in \D,
\end{equation}
therefore, we just need to estimate $k_{\mu, n}(z,z)$ from below. Let $\mathcal{P}_n$ denote the $n$-dimensional space of polynomials of degree at most $n-1$. Since $k_{\mu , n}$ is the reproducing kernel, we have 
\begin{align*}
k_{\mu , n}(z,z) 
&= \|k_{\mu , n}(\cdot, z)\|_{L^2(\mu)}^{2} \\
&= \sup\{|(p, k_{\mu , n}(\cdot, z))_{L^2(\mu)}|^2,\;\; p \in \mathcal{P}_n,\; \|p\|_{L^2(\mu)} \le 1\}\\
&= \sup\{|p(z)|^2,\;\; p \in \mathcal{P}_n,\; \|p\|_{L^2(\mu)} \le 1\}\\
&= \sup\left\{\frac{|p(z)|^2}{\|p\|_{L^2(\mu)}^2},\;\; p \in \mathcal{P}_n\right\}.
\end{align*}
Substituting $p(\xi) = \frac{1- \bar z^{n} \xi^n}{1 - \bar z \xi}$ into the last supremum, we get
$$
k_{\mu , n}(z,z) \ge \left(\frac{1 - |z|^{2n}}{1-|z|^2}\right)^2 \cdot \left(\int_{\T}\left|\frac{1- \bar z^{n} \xi^n}{1 - \bar z \xi}\right|^2\,d\mu(\xi)\right)^{-1}.
$$
Using $|z|^n \to 0$ and $|z| \to 1$ for $z \in \Gamma_{\zeta, n}(\delta_n, \Delta_n)$ when $n \to \infty$, we obtain
\begin{align}
k_{\mu , n}(z,z) 
&\ge \left(\frac{1}{1-|z|^2}\right)^2 \cdot \left(\int_{\T}\frac{1}{|1 - \bar z \xi|^2}\,d\mu(\xi)\right)^{-1}\cdot (1 +o(1)) \notag\\
&= \frac{1}{1-|z|^2} \cdot \left(\int_{\T}\frac{1-|z|^2}{|1 - \bar z \xi|^2}\,d\mu(\xi)\right)^{-1}\cdot (1 +o(1)) \notag\\
&= \frac{1}{1-|z|^2} \cdot \Pp(\mu, z)^{-1}\cdot(1 +o(1)) \notag\\
&= \frac{1}{1-|z|^2} \cdot |D_{\mu}(z)|^{-2}\cdot(1 +o(1)). \label{eq130-2}
\end{align}
In the last line we have used the fact that $\Pp(\mu, z) |D_{\mu}(z)|^{-2} = e^{\K(\mu, z)}$, $ z\in \D$, because $D_\mu$ is an outer function. In particular, we have $\Pp(\mu, z) = |D_{\mu}(z)|^2(1 + o(1))$ in $\Gamma_{\zeta, n}(\delta_n, \Delta_n)$ provided \eqref{eq21} holds. 
Combining \eqref{eq129}, \eqref{eq130-2}, we obtain \eqref{eq777} and complete the proof. \qed

\medskip

\begin{Lem}\label{7p2}
Let $\mu \in \szc$ and $\zeta \in \T$ be such that $D_\mu$ has a non-zero finite non-tangential limit at $\zeta$. If \eqref{eq21} holds at $\zeta$, then for every $A\ge 0$ we have 
$$
\lim_{n \to \infty} \sup_{\substack{z_1, z_2 \in \Gamma_{\zeta},\\ d_H(z_1, z_2) \le A}}\bigl|r_n(z_1, z_2)\bigl| = 0$$
for the functions $r_n$ defined in \eqref{universality0-thm4-bis}.
\end{Lem}
\beginpf Let $\Gamma'_{\zeta, n}(\delta)$ be defined by \eqref{eq872}, and let $\Gamma_{\zeta, n}(\delta) = \Gamma_{\zeta}(\delta) \setminus \Gamma'_{\zeta, n}(\delta)$. As in the proof of Lemma  \ref{l51}, it follows from Theorem 1.1 in \cite{Bes24} that there exists a sequence $\delta_n \to \infty$, $\delta_n = o(n)$,  such that
$$
\sup_{z_1, z_2 \in \Gamma'_{\zeta, n}(\delta_n)}\bigl|r_n(z_1, z_2)\bigl| \to 0, \quad n \to \infty,
$$
and we only need to check that
$$
\sup_{\substack{z_1, z_2 \in \Gamma_{\zeta, n}(\delta_n),\\ d_H(z_1, z_2) \le A}}\bigl|r^*_n(z_1, z_2)\bigl| \to 0, \quad n \to \infty, 
$$
for $r^*_n$ defined by 
$$
k_{\mu, n}(z_1, z_2) =  \frac{\ov{D_{\mu}^{-1}(z_2)}D_{\mu}^{-1}(z_1)}{1 - \bar z_2 z_1}\bigl(1 + r^*_n(z_1, z_2)\bigr).
$$
For this we will use formulas \eqref{eq145}, \eqref{eq146}. Cauchy-Schwarz inequality gives 
\begin{equation}\label{eq76}
\left|\frac{\ov{D_{\mu}^{-1}(z_2)}D_{\mu}^{-1}(z_1)}{1 - \bar z_2 z_1} - k_{\mu, n}(z_1, z_2) \right|^2
\le \big( \sum_{k = n}^{\infty}|\phi_k(z_1)|^2 \big) \cdot \big( \sum_{k = n}^{\infty}|\phi_k(z_2)|^2 \big). 
\end{equation}
By Lemma \ref{l51}, for every $z \in \Gamma_{\zeta, n}(\delta_n)$ we have 
$$
\sum_{k = n}^{\infty}|\phi_k(z)|^2  = \frac{|D_{\mu}^{-1}(z)|^2}{1 - |z|^2} - k_{\mu, n}(z, z) = \frac{|D_{\mu}^{-1}(z)|^2}{1 - |z|^2}\tilde R_n(z),
$$
for some $\tilde R_n$ such that $\sup_{z \in \Gamma_{\zeta, n}(\delta_n)}|\tilde R_n(z)| \to 0$ as $n \to \infty$. Now \eqref{eq76} gives 
\begin{align*}
\left|\frac{\ov{D_{\mu}^{-1}(z_2)}D_{\mu}^{-1}(z_1)}{1 - \bar z_2 z_1} - k_{\mu, n}(z_1, z_2)\right| 
&= \frac{|D_{\mu}^{-1}(z_1)D_{\mu}^{-1}(z_2)|}{\sqrt{1-|z_1|^2}\sqrt{1-|z_2|^2}}\sqrt{|\tilde R_n(z_1) \tilde R_n(z_2)|},\\
&\asymp  \left|\frac{\ov{D_{\mu}^{-1}(z_2)}D_{\mu}^{-1}(z_1)}{1 - \bar z_2 z_1}\right|\sqrt{|\tilde R_n(z_1) \tilde R_n(z_2)|},
\end{align*}
where we used 
$$
\sqrt{1-|z_1|^2}\sqrt{1-|z_2|^2} \asymp |1 - \bar z_2 z_1|
$$
for $z_1, z_2 \in \Gamma_{\zeta}$ such that $d_{H}(z_1, z_2) \lesssim 1$.  In particular, we have 
$$
|r^*_n(z_1, z_2)| \lesssim \sqrt{|\tilde R_n(z_1) \tilde R_n(z_2)|},
$$ 
and the result follows from Lemma \ref{l51}. \qed

\medskip

\begin{Lem}\label{7p3}
Let $\mu \in \szc$, $b_n = \phi_n/\phi_n^*$ and $f_n$ be the Schur functions of $\mu$ defined in \eqref{eq41}. Then for every $n\ge 1$ we have
\begin{equation}\label{eq132}
|b_n(z)| \lesssim |f_{n-1}(z)| + |f_{n-1}(0)| + |z|^n + \sqrt{e^{\K(\mu, z)} - 1},
\end{equation}
provided $\max \{|f_{n-1}(z)| , |f_{n-1}(0)|\} \leq 1/2$. In particular, if \eqref{eq21} holds, we have 
\begin{equation}\label{eq133}
\lim_{n\to \infty}\sup_{z \in \Gamma_{\zeta, n}(\delta_n)}\frac{1}{n}\sum_{k=0}^{n-1} |b_k(z)| = 0,
\end{equation}
for every $\delta_n \to \infty$, $\delta_n = o(n)$.
\end{Lem}
\beginpf For $z \in \D$, set $\tilde b_n = \tilde\phi_{z,n}/\tilde\phi_{z,n}^*$, see \eqref{eq44}, \eqref{eq44star}. Note that the definition of $\tilde b_n$ depends on the choice of the reference point $z$. Thus, $\tilde b_{z,n}$ would be a more accurate notation for this function. However, we will deal with $\tilde b_n$ only at the same point $z$, therefore, the short notation $\tilde b_n$ should not lead to misunderstanding. By definition, we have 
\begin{align*}
|\tilde b_n(z) - b_n(z)| 
&= \left|\frac{z b_{n-1}(z) - \ov{f_{n-1}(z)}}{1 - z b_{n-1}(z)f_{n-1}(z)} - \frac{z b_{n-1}(z) - \ov{f_{n-1}(0)}}{1 - z b_{n-1}(z)f_{n-1}(0)}\right| \\
&\le \frac{2|f_{n-1}(0)| + 2|f_{n-1}(z)|}{(1 - |f_{n-1}(0)|)(1 - |f_{n-1}(z)|)}.
\end{align*}
It follows that $|\tilde b_n(z) - b_n(z)| \lesssim |f_{n-1}(0)| + |f_{n-1}(z)|$ if $\max \{|f_{n-1}(z)|, |f_{n-1}(0)| \} \leq 1/2$. Moreover, we have 
$$
|\tilde b_n(z) - \alpha_z z^n|^2 \lesssim e^{\K(\mu, z)} - 1, \qquad |\alpha_z| = 1,
$$
by Lemma 2.5 in \cite{Bes24} and \eqref{eq66}. Then \eqref{eq132} follows. To prove \eqref{eq133}, we estimate
\begin{align*}
\frac{2}{n}\sum_{k= n/2}^{n} |b_k(z)| 
\lesssim &\sqrt{e^{\K(\mu, z)} - 1} + \frac{2}{n(1-|z|)} + \\
&+ \frac{2|\{n/2 \le k \le n:\;|f_{k-1}(z)| > 1/2\}|}{n} + \etat_n(z)   \\
\lesssim &\sqrt{e^{\K(\mu, z)} - 1} + \frac{2}{n(1-|z|)} + \etat_n(z),
\end{align*}
where $\etat_n(z)$ is defined in \eqref{etatn}. Then 
$$
\lim_{n\to \infty}\sup_{z \in \Gamma_{\zeta, n}(\delta_n, \Delta_n)}\frac{2}{n}\sum_{k= n/2}^{n} |b_k(z)| = 0, 
$$
for any sequences $\Delta_n \to 1$, $\delta_n \to \infty$ with $\delta_n = o (n)$, by \eqref{eq21} and Lemma \ref{l32} (for the definition of 
$\Gamma_{\zeta, n}(\delta_n, \Delta_n)$, see \eqref{eq871}). At the same time, we have $\phi_n^* \to D_{\mu}^{-1}$, $\phi_n \to 0$ uniformly on compact subsets of $\D$ by Szeg\H{o} theorem, hence 
$$
\lim_{n\to \infty}\sup_{z \in \Gamma_{\zeta}(\Delta_n)}\frac{2}{n}\sum_{k= n/2}^{n} |b_k(z)| = 0
$$
for some sequence $\Delta_n \to 1$. Since $\Gamma_{\zeta}(\Delta_n) \cup \Gamma_{\zeta, n}(\delta_n, \Delta_n) = \Gamma_{\zeta, n}(\delta_n)$, see \eqref{eq871}-\eqref{eq873}, this implies that 
$$
\mathcal{B}_n := \sup_{z \in \Gamma_{\zeta  ,n}(\delta_n)}\frac{2}{n}\sum_{k= n/2}^{n} |b_k(z)| \to 0, \quad n \to \infty .  
$$
Now given a positive integer $N$, pick the positive integer $n$ with $2^{n-1} \leq N < 2^n$. Then 
$$
\frac{1}{N} \sum_{k=1}^N |b_k (z)| \leq \frac{2}{2^n} \sum_{j=0}^n 2^j \mathcal{B}_{2^j}
$$
which tends to $0$ as $N \to \infty$. This implies \eqref{eq133}. \qed

\medskip

\begin{Lem}\label{l-prod}
For every $\eps>0$ and every complex numbers $x_k$, $y_k$, such that 
$$
\frac{1}{n}\sum_{k=0}^{n-1}\left(|x_k - 1| + \bigl||x_k|^2 - 1\bigr| + |y_k - 1| + \bigl||y_k|^2 - 1\bigr|\right) < \eps
$$
we have
$$
\frac{1}{n}\sum_{k=0}^{n-1}|x_k y_k - 1| < 6\eps.
$$
\end{Lem}
\beginpf Define $I_{n,s} = \{0\le k\le n-1 \colon |x_k| + |y_k| \le 4\}$ and $I_{n,b} = \{0\le k\le n-1 \colon |x_k| + |y_k| > 4\}$. We have
\begin{align*}
\frac{1}{n}\sum_{k\in I_{n,s}}|x_k y_k - 1| 
&\le \frac{1}{n}\sum_{k\in I_{n,s}}|x_k - 1| |y_k| + \frac{1}{n}\sum_{k\in I_{n,s}}|y_k - 1|\\
&\le \frac{4}{n}\sum_{k=0}^{n-1}|x_k - 1| + \frac{1}{n}\sum_{k=0}^{n-1}|y_k - 1| < 4\eps.
\end{align*}
On the other hand, for $k \in I_{n,b}$ we have either $|x_k| > 2$ or $|y_k| > 2$, hence 
\begin{align*}
|x_k y_k - 1| 
&\le 1 + \max(|x_k|^2, |y_k|^2), \\
&\le \max(|x_k - 1|, |y_k - 1|) + 2\max(\bigl||x_k|^2 - 1\bigr|, \bigl||y_k|^2 - 1\bigr|),\\
&\le 2(|x_k - 1| + \bigl||x_k|^2 - 1\bigr| + |y_k - 1| + \bigl||y_k|^2 - 1\bigr|),
\end{align*}
and our assumption implies the bound $n^{-1} \sum_{k\in I_{n,b}}|x_k y_k - 1| < 2\eps$. \qed
\begin{Lem}\label{7p5}
Let $\mu \in \szc$ and $\zeta \in \T$ be such that $D_\mu$ has a non-zero finite non-tangential limit at $\zeta$. If \eqref{eq22} holds at $\zeta$, then we have 
$$
\lim_{n \to \infty}\sup_{z_1, z_2 \in \Gamma_{\zeta}}\frac{1}{n}\sum_{k=0}^{n-1}\bigl|r_k(z_1, z_2)\bigl|= 0
$$
for the functions $r_n$ defined in \eqref{universality0-thm4-bis}.
\end{Lem}
\beginpf Given a sequence $\delta_n \to \infty$, $\delta_n = o(n)$, let
\begin{align*}
\Omega_n(\delta_n) &= \{(z_1, z_2): \; z_1 \mbox{ or } z_2 \in \Gamma_{\zeta, n}(\delta_n)\},\\
\Omega'_n(\delta_n) &= \{(z_1, z_2): \; z_1, z_2 \in \Gamma_{\zeta}\setminus \Gamma_{\zeta, n}(\delta_n)\}.
\end{align*}
It follows from Theorem 1.1 in \cite{Bes24} that there exists a sequence $\delta_n \to \infty$, $\delta_n = o(n)$, such that the numbers
$$
\mathcal{E}_{1,n} = \sup_{(z_1, z_2) \in \Omega'_n(\delta_n)}|r_n(z_1, z_2)|
$$
tend to $0$ as $n \to \infty$. Therefore, we need to estimate $r_n(z_1,z_2)$ for the pairs $(z_1, z_2)$ in $\Omega_n(\delta_n)$. Set $b_n = \phi_n/\phi_n^*$. Using the formula \eqref{reprokernel} written in the form
$$
k_{\mu, n}(z_1, z_2) = \phi_n^*(z_1)\ov{\phi_n^*(z_2)}\frac{1 - b_n(z_1)\ov{b_n(z_2)}}{1 - \bar z_2 z_1}, \qquad b_n = \frac{\phi_n}{\phi^*_n},
$$
we express funtions $r_n$ in \eqref{universality0-thm4-bis} in the form
\begin{align*}
r_n(z_1, z_2) 
= &\frac{(\ov{D_{\mu}(z_2)}D_{\mu}(z_1)\phi_n^*(z_1)\ov{\phi_n^*(z_2)} - 1)(1 - b_n(z_1)\ov{b_n(z_2)})}{1 - \bar z_2^{n} z_1^{n}} + \\ 
&+  \frac{ \bar z_2^{n} z_1^{n} - b_n(z_1)\ov{b_n(z_2)}}{1 - \bar z_2^{n} z_1^{n}},
\end{align*}
by means of an algebraic manipulation. For $(z_1, z_2) \in \Omega_n(\delta_n)$ and $n/2 \le k \le n$, we have $|1 - \bar z_2^{k} z_1^{k}| \asymp 1$. Consider the numbers
\begin{align*}
\mathcal{E}_{2,n} 
= &\sup_{(z_1, z_2) \in \Omega_n(\delta_n)}\frac{2}{n}\sum_{n/2\le k \le n}|r_k(z_1, z_2)| \\
\lesssim &\sup_{z_1, z_2 \in \Gamma_{\zeta}}\frac{2}{n}\sum_{n/2\le k \le n}|\ov{D_{\mu}(z_2)}D_{\mu}(z_1)\phi_k^*(z_1)\ov{\phi_k^*(z_2)} - 1| + \\
&+ \sup_{(z_1, z_2) \in \Omega_n(\delta_n)}\frac{2}{n}\sum_{n/2\le k \le n} |z_1^k \bar z_2^k| + \sup_{(z_1, z_2) \in \Omega_n(\delta_n)}\frac{2}{n}\sum_{n/2\le k \le n}|b_k(z_1) b_k(z_2)|. 
\end{align*}
Theorem \ref{thm1}, Theorem \ref{thm2} and Lemma \ref{l-prod} imply that  
$$
\lim_{n\to \infty}\sup_{z_1, z_2 \in \Gamma_{\zeta}}\frac{1}{n}\sum_{0\le k \le n-1}|\ov{D_{\mu}(z_2)}D_{\mu}(z_1)\phi_k^*(z_1)\ov{\phi_k^*(z_2)} - 1| = 0.
$$
The last relation, the fact that $\delta_n \to \infty$, and Lemma \ref{7p3} give $\mathcal{E}_{2,n} \to 0$ (recall that $b_n$ is a Blaschke product, so $|b_n| < 1$ in $\D$ and one can apply Lemma \ref{7p3} to the point $z$ ($=z_1$ or $z_2$) lying in $\Gamma_{\zeta, n}(\delta_n)$ for estimating $\frac{2}{n}\sum_{n/2\le k \le n}|b_k(z_1) b_k(z_2)|$). Combining this fact with $\mathcal{E}_{1,n} \to 0$, we see that the quantities
$$
\mathcal{E}_{3,n} =\sup_{z_1, z_2 \in \Gamma_{\zeta}}\frac{2}{n}\sum_{n/2\le k \le n}|r_k(z_1, z_2)|
$$
also tend to zero as $n \to \infty$. At the same time, for every integer $\ell \ge 1$ we have
$$
\sup_{z_1, z_2 \in \Gamma_{\zeta}}\frac{1}{2^\ell}\sum_{1\le k \le 2^\ell}|r_k(z_1, z_2)| \le \frac{\mathcal{E}_{3,2^\ell}}{2} + \frac{\mathcal{E}_{3,2^{\ell-1}} }{4}+ \frac{\mathcal{E}_{3,2^{\ell-2}}}{8} + \ldots +\frac{\mathcal{E}_{3,2}}{2^{\ell}}.
$$
This estimate and the fact that $\mathcal{E}_{3,n} \to 0$ yield the statement. \qed

\medskip

\begin{Lem}\label{7p6}
Let $\mu \in \szc$ and $\zeta \in \T$ be such that $D_\mu$ has a non-zero finite non-tangential limit at $\zeta$. If \eqref{eq23} holds at $\zeta$, then we have 
$$
\lim_{n \to \infty} \sup_{z_1, z_2 \in \Gamma_{\zeta}}\bigl|r_n(z_1, z_2)\bigl| = 0
$$
for the functions $r_n$ defined in \eqref{universality0-thm4-bis}.
\end{Lem}
\beginpf As in the proof of Lemma \ref{7p3}, for each $z \in \D$ we set $b_n = \phi_{n}/\phi_{n}^*$, $\tilde b_n = \tilde\phi_{z,n}/\tilde\phi_{z,n}^*$, so that $|\tilde b_n(z) - b_n(z)| \lesssim |f_{n-1}(0)| + |f_{n-1}(z)|$ if $|f_{n-1}(z)|$, $|f_{n-1}(0)|$ are less than $1/2$. By Lemma \ref{lemChain}, the last assumption holds for all $z \in \Gamma_\zeta$ if $n$ is sufficiently large. On the other hand, we have 
$$
|\tilde b_n(z) - \alpha_ z z^n|^2 \lesssim e^{\K(\mu, z)} - 1
$$
by Lemma 2.5 in \cite{Bes24}. Here, 
$$
\alpha_ z = \frac{\ov{\tilde\phi_{z,n}^*(z)}}{\tilde\phi_{z,n}^*(z)} = \frac{\ov{\phi_{n}^*(z)}}{\phi_{n}^*(z)}(1+o(1)) = \frac{\ov{D_{\mu}(z)^{-1}}}{D_{\mu}(z)^{-1}}(1+o(1)), 
$$
with uniform in $\Gamma_\zeta$ remainders $o(1)$ by Lemma \ref{lemChain} and Theorem \ref{thm3}. Since $D_{\mu}$ has a finite non-zero limit at $\zeta$, the quantity $\alpha_ z$ tends to some $\alpha \in \T$ when $z$ approaches $\zeta$ non-tangentially. We also have $b_n(z) \to 0$ uniformly on compact subsets of $\D$ by the Szeg\H{o} theorem. It follows that 
$$
\lim_{n\to \infty}\sup_{z_1, z_2 \in \Gamma_{\zeta}}|b_n(z_1)\ov{b_n(z_2)}- \bar z_2^n z_1^n| = \lim_{n\to \infty}\sup_{z_1, z_2 \in  \Gamma_{\zeta}}|\ov{\alpha z_2^n} \cdot \alpha z_1^n - \bar z_2^n z_1^n| = 0.
$$
It remains to use the fact that 
$$
k_{\mu, n}(z_1, z_2) = \phi_n^*(z_1)\ov{\phi_n^*(z_2)}\frac{1 - b_n(z_1)\ov{b_n(z_2)}}{1 - \bar z_2 z_1}, 
$$
and the uniform asymptotic relation \eqref{noou} in Theorem \ref{thm3}. \qed

\medskip

\noindent{\bf Proof of Theorem \ref{thm4}.} See Lemma \ref{7p2}, Lemma \ref{7p5}, and Lemma \ref{7p6}. \qed

\medskip

\section{Proofs of Propositions \ref{p31}, \ref{p21} and \ref{p11}}

\noindent {\bf Proof of Proposition \ref{p11}.} Denote $u = \log w$. Since the value $\K(\mu, z)$ of the entropy of $\mu$ at any point $z \in \D$ is invariant under multiplication of $\mu$ by a constant, we can assume that $u(\zeta) = 0$. Since $u$ is bounded at a neighborhood of $\zeta$ and condition \eqref{eq23} is local (provided that $u \in L^1(\T)$ as we always assume), we can suppose that $u$ is bounded on $\T$ and vanishes on $\T \setminus I_0$. We have 
$$
\Pp(w, z) = \Pp(e^{u}, z) =  1 + \Pp(u, z) +  O\left(\Pp(u^{2}, z)\right), 
$$
and, since $\Pp(u, z)^2 \le \Pp(u^2, z)$, 
$$
\log\Pp(w, z) = \Pp(u, z) +  O\left(\Pp(u^{2}, z)\right).    
$$
At the same time, we have $\Pp(\log w, z) = \Pp(u, z)$ by definition. Thus, we have
$$
\K(\mu, z) \lesssim \Pp(u^{2}, z), \quad z \in \D. 
$$
Standard estimates of the Poisson kernel give (recall that $u=0$ on $\T\setminus I_0$)
$$
\Pp(u^2, z_n (\zeta)) \lesssim \sum_{k=0}^n 2^{2k - n} \int_{I_k} u^2 \,dm. 
$$
Thus,
$$
\sqrt{\K(\mu , z_n ( \zeta))} \lesssim \sum_{k=0}^n 2^{k - n/2} \left( \int_{I_k} u^2 \,dm \right)^{1/2},
$$
and hence
$$
\sum_{n=0}^\infty \sqrt{\K(\mu , z_n (\zeta))} \lesssim  \sum_{k=0}^\infty   2^{k/2} \left( \int_{I_k} u^2 \,dm \right)^{1/2}.
$$
The result follows. \qed

\medskip

\noindent {\bf Proof of Proposition \ref{p21}.} As in the proof of Proposition \ref{p11}, we denote $u =\log w$, assume that $u(\zeta) =0$ and $u$ is bounded on $\T$, and then
$$
\K(\mu, z) \lesssim \Pp(u^{2}, z), \quad z \in \D. 
$$
We see that
$$
\sum_{n=0}^\infty \K(\mu , z_n (\zeta)) \lesssim \sum_{n=0}^\infty \Pp(u^{2}, z_n (\zeta)),
$$
and one only needs to use the estimate
$$
\sum_{n=0}^\infty \frac{1- |z_n (\zeta)|^2}{|1- z_n (\zeta)\overline{\xi} |^2} \lesssim \int_0^1 \frac{dr}{|\xi- r \zeta|^2} \lesssim \frac{1}{|\xi - \zeta|} , \quad \xi , \zeta \in \T,
$$
to complete the proof. \qed

\medskip

\noindent {\bf Proof of Proposition \ref{p31}.} We need to show that condition \eqref{eq22} holds for almost every point $\zeta \in \T$. Arguing as in the proof of Proposition \ref{p11}, we arrive at the estimate
$$
\K(\mu, r \zeta) \lesssim \int_{\T} ( u(\xi) - u(\zeta))^2 \frac{1- r^2}{|\xi -  r \zeta|^2}\,dm(\xi) , \qquad 0\le r<1, \quad  \zeta \in \T. 
$$
Denote by $A$ the Lebesgue measure on $\C$ normalized by $A(\D) = 1$. We have
\begin{align*}
& \int_{\D} \frac{\K(\mu , z)}{1-|z|^2}\, dA(z) 
= \int_{\T} \int_0^1 \frac{\K(\mu, r \zeta) r}{1- r^2} dr \, dm (\zeta)\lesssim\\
& \lesssim \int_{\T} \int_{\T} (u(\xi) - u(\zeta))^2 \int_0^1 \frac{dr}{|\xi -  r \zeta|^2} \,dm(\xi) \, dm (\zeta).     
\end{align*}
Since
$$
\int_0^1 \frac{dr}{|\xi -  r \zeta|^2} \lesssim \frac{1}{|\xi - \zeta|}, \quad \xi , \zeta \in \T , 
$$
we deduce that 
$$
\int_{\D} \frac{\K(\mu , z)}{1-|z|^2} \, dA(z) \lesssim \int_{\T} \int_{\T} \frac{( u(\xi) - u(\zeta))^2}{|\xi - \zeta|}  \,dm(\xi) \, dm (\zeta).     
$$
Denoting $\zeta = e^{2 \pi i \theta}$ and $\xi = e^{2 \pi i (\theta + h)}$, the last double integral rewrites as
$$
 \int_{0}^{1}\frac{1}{|1 - e^{2 \pi i h}|} \int_{0}^{1}(u(e^{2 \pi i \theta}) - u (e^{2 \pi i (\theta + h)}))^2  \,d\theta \, dh. 
$$
The Fourier coefficients of the real-valued function $\theta \mapsto u(e^{2 \pi i \theta}) - u (e^{2 \pi i (\theta + h)})$ with respect to the basis 
$\{e^{2 \pi i k \theta}\}_{k \in \Z}$ in $L^2[0,1]$ are $(1- e^{2 \pi i k h})\hat u(k)$, $k \in \Z$. 
Then, Parseval's identity gives
\begin{align*}
    \int_{\T} \int_{\T}  \frac{( u(\xi) - u(\zeta))^2}{|\xi -  \zeta|} \, dm (\xi) \, dm (\zeta) & =   \sum_{k=- \infty}^{+ \infty} |\hat u(k)|^2 \int_{0}^{1} \frac{|1- e^{2 \pi i k h}|^2}{|1 - e^{2 \pi i h}|} dh \\ & \asymp  \sum_{k=- \infty}^{+ \infty} |\hat u(k)|^2 \log (2+ |k|) . 
\end{align*}
Hence assumption \eqref{eq33} implies
$$
\int_{\D} \frac{\K(\mu , z)}{1-|z|^2}\,dA(z) < \infty. 
$$
Fubini theorem then gives
\begin{align*}
 \int_{\T} \int_{\Gamma_\zeta} \frac{\K (\mu , z)}{(1- |z|^2)^2} dA(z) \, dm (\zeta) & = \int_{\D} \frac{\K (\mu , z)}{(1- |z|^2)^2} m (\{ \zeta \in \T : z \in \Gamma_\zeta \}) d A(z) \\ & \lesssim \int_{\D} \frac{\K(\mu , z)}{1-|z|^2}\,d A(z),    
\end{align*}
from where we deduce that 
$$
\int_{\Gamma_\zeta} \frac{\K(\mu , z)}{(1- |z|^2)^2}\,dA(z) < \infty, 
$$
for almost every $\zeta \in \T$. Hence condition \eqref{eq22} holds for almost every $\zeta \in \T$. \qed

\medskip

\noindent {\bf Acknowledgment.} The authors are grateful to S.\,Denisov for drawing their attention to the paper \cite{Ul92}.

\smallskip

\noindent The work of RB is supported by the Slovenian Research Agency ARIS (grants J1-70033 and P1-0291). The work of AN is supported by the Spanish Ministerio de Ciencia e Innovaci\'on (grant PID2021-123151NB-I00), by the Generalitat de Catalunya (grant 2021 SGR 00071), and by the Spanish Research
Agency through the María de Maeztu Program (CEX2020-001084-M).

\medskip

\bibliographystyle{plain} 	
\bibliography{bibfile} 
\end{document}